\documentclass[12pt]{article}
\usepackage{amsmath, amsthm, amssymb}
\numberwithin{equation}{section}

\newtheorem{lemma}{Lemma}

\begin{document}
\author{Ajai Choudhry}
\title{A new method of solving quartic and \\
higher degree diophantine equations}
\date{}
\maketitle

\frenchspacing

\begin{abstract} 
In this  paper we present a new method of solving certain quartic and higher degree homogeneous polynomial diophantine  equations  in four variables.  The method can also  be  extended  to solve  simultaneous  homogeneous polynomial diophantine equations, in five or more variables, with one of the equations being of degree $\geq 4$. We show that, under certain conditions,  the method yields an arbitrarily large number of integer solutions of such diophantine equations and  diophantine systems,   some examples being a sextic equation in four variables, a tenth degree   equation in six variables, and two simultaneous equations of degrees four and six in  six variables. The method of solving these homogeneous equations also simultaneously yields  arbitrarily many rational solutions of  certain related  nonhomogeneous equations of high degree. In contrast to existing methods, we obtain the arbitrarily  large number of  solutions without finding a parametric solution of the equations under consideration  and without relating the solutions to rational points on an elliptic curve of positive rank. It appears from   the examples  given in the paper  that there may exist projective varieties on which there are an arbitrarily large number of integer points and on which  a curve of genus 0  or 1 does not exist.  
\end{abstract}


\hspace*{1.0in} equations; rational points; projective varieties; \\
\hspace*{1.0in} composition of forms.

\section{Introduction}\label{Intro}

This paper is  concerned with quartic and higher degree homogeneous polynomial diophantine equations of the type,
\begin{equation}
f(x_1,\,x_2,\,x_3,\,x_4)=0, 
\label{gendeq}
\end{equation}
where $f(x_1,\,x_2,\,x_3,\,x_4)$ is a form, with integer coefficients, in the 4   variables $x_1,\,\,x_2,\,x_3,x_4$, as well as with simultaneous diophantine equations  of the type,
\begin{equation}
f_r(x_1,\,x_2,\,\ldots,\,x_{n})=0,\;r=1,\,2,\,\ldots,\,k,\label{vsysgen0} \\
\end{equation}
where $f_r(x_i)$ are forms, with integer coefficients,  in the $n$ independent variables $x_1,\,x_2,\,\ldots,\,x_{n}$ where $n \geq 5$,    $k < n-2$ and the degree of at least one of the forms $f_r(x_i)$ is $\geq 4$.

If there exist infinitely many integer solutions of Eq.~\eqref{gendeq} or of the diophantine system \eqref{vsysgen0},   till now there are primarily  just two ways  of obtaining  these infinitely many  solutions: either we find a  solution in terms of one or more independent, arbitrary integer parameters, or we find a solution in terms of two parameters $X$ and $Y$ where the ordered pair $(X,\,Y)$ represents a rational point on an elliptic curve defined over $\mathbb{Q}$, and such that the elliptic curve  has rank $\geq 1$. Since the elliptic curve has positive rank, we can find infinitely many rational points on the elliptic curve  by using the group law, and thus obtain  infinitely many rational solutions of Eq.~\eqref{gendeq} or of the diophantine system \eqref{vsysgen0}. On appropriate scaling, these infinitely many rational solutions yield infinitely many primitive integer solutions of Eq.\eqref{gendeq} or of the simultaneous equations \eqref{vsysgen0}.

In this paper we first show that when Eq.~\eqref{gendeq} satisfies certain properties, we may  obtain  a sequence of integer solutions of quartic and higher degree equations of type \eqref{gendeq} by a new iterative method. In contrast to existing methods, the method described in this paper   neither involves finding a parametric solution nor does it require relating the integer solutions to rational points on an elliptic curve of positive rank. In fact, we show how to construct equations of arbitrarily high degree with the desired properties, and give examples of equations for which we obtain an infinite sequence of integer solutions.   While the  octic and higher degree equations, that we could solve by the new method, are such that there is always  a curve of genus 0 or 1 on the surface defined by the equation under consideration, we show that there exist quartic and sextic equations of type \eqref{gendeq} with infinitely many integer solutions  and for which it seems that there is no particular reason that there should necessarily exist a curve of genus 0 or 1 on the surface defined by the equation.

As in the case of the single diophantine equation \eqref{gendeq}, when the diophantine system \eqref{vsysgen0} possesses certain properties, an extension of the new method may be applied to obtain a sequence of integer solutions of \eqref{vsysgen0}. We show how to construct such diophantine systems in which one of  the  equations is of arbitrarily high degree. This leads to an example of a diophantine equation of arbitrarily high degree $2d$ in $2d$ variables for which we obtain an arbitrarily long sequence of integer solutions. While this sequence is expected to consist of infinitely many distinct integer solutions, this remains to be proved.  

We give a couple of  illustrative examples of  diophantine systems for which an arbitrarily large number of  integer solutions may be obtained by the method described in this paper. For instance,   we  obtain arbitrarily many solutions of   a diophantine system consisting of one equation of degree 10 in 8 variables and two linear equations.  This leads to an example of a single diophantine equation of degree 10 in 6 variables as well as to an example of a pair of simultaneous equations of degrees four and six in 6 variables, and for both of these examples, we obtain an arbitrarily large number of  integer solutions.  It appears that  there exist  diophantine systems  with  an arbitrarily large number of  integer solutions and for which it is unlikely that there exists  a curve of genus 0 or 1  on the projective variety  defined by the simultaneous equations comprising the diophantine system. 

As a related result, when we obtain an arbitrarily large number of  integer solutions of an equation of type \eqref{gendeq} or a diophantine system of type \eqref{vsysgen0}, we  also simultaneously get a nonhomogeneous equation of high degree with an arbitrarily large number of  rational solutions.

In Section 2 we give an overview of the new method  of solving diophantine equations of type \eqref{gendeq}. In Sections 3, 4 and 5, we show how to construct  examples of quartic, sextic and higher degree equations that may be solved by this method. In Section 6 we give an overview of an extension of the  method  for solving simultaneous diophantine equations in several variables, and show how to construct  diophantine systems that may be solved by this method.  In Sections 7 and 8, we give examples of such diophantine systems with  an arbitrarily large number of  integer solutions. We conclude the paper with some open problems and remarks regarding such diophantine equations.  

\section{An overview of the new method}\label{newmethod}

In this section we describe, in general terms, the process of generating a sequence of integer solutions of certain quartic and higher  degree homogeneous diophantine equations  in  four variables. 

\subsection{}\label{properties}

We now consider the  diophantine equation,
\begin{equation}
f(x_1,\,x_2,\,x_3,\,x_4)=0. \label{gendeq4var}
\end{equation}
where $f(x_1,\,x_2,\,x_3,\,x_4)$ is a  quaternary form of degree $d \geq 4$ in four independent variables $x_1,\,x_2,\,x_3,\,x_4$.  Since Eq.~\eqref{gendeq4var} is homogeneous, if $(\alpha_1,\,\alpha_2,\,\alpha_3,\,$ $\alpha_4) $
is any rational solution of \eqref{gendeq4var}, then  $(k\alpha_1,\,k\alpha_2,\,k\alpha_3,$ $k\alpha_4), \; k \in \mathbb{Q} \setminus \{0\}$ is also a rational solution of \eqref{gendeq4var}. All such solutions will be considered equivalent, and they all represent the same point on the surface \eqref{gendeq4var}. When we refer to more than one  solution of \eqref{gendeq4var}, we mean solutions that are not equivalent to each other. It is clear that any rational solution of Eq.~\eqref{gendeq4var} yields, on appropriate scaling, a solution  in integers. Thus, it suffices to find rational solutions of \eqref{gendeq4var}. 

If there exist integers $\alpha_1,\,\alpha_2,\,\alpha_3,\,\alpha_4$, not all 0,  such that $f(\alpha_i)=0$,  we will refer to the quadruple $(\alpha_1,\,\alpha_2,\,\alpha_3,\,\alpha_4)$ both  as a solution of Eq.~\eqref{gendeq4var} and  as a point on the surface defined by Eq.~\eqref{gendeq4var}.

We assume that Eq.~\eqref{gendeq4var} satisfies the following properties:

\noindent ${\bf D_1:}$ At least one integer solution $(\alpha_1,\,\alpha_2,\,\alpha_3,\,\alpha_4)$ of the diophantine equation \eqref{gendeq4var} is known such that both $\alpha_1, \,\alpha_2$ are not simultaneously $0$ and similarly,  both $\alpha_3, \,\alpha_4$ are not simultaneously $0$.

\noindent  ${\bf D_2:}$ The  form $f(x_i)$ satisfies the condition $f(x_1,\,x_2,\,x_3,\,x_4)$ $=f(-x_1,\,x_2,\,x_3,$ $x_4)$. 

\noindent ${\bf D_3:}$ There exist bilinear forms $B_1(m_1,\,m_2,\,x_1,\,x_2)$ and $B_2(m_1,\,m_2,\,x_1,\,x_2)$  with rational coefficients satisfying the following conditions:

(i) for all rational numerical values of $x_1$ and $x_2$ such that $(x_1,\,x_2) \neq (0,\,0)$, the  forms $B_j(m_1,\,m_2,\,x_1,\,x_2),\;j=1,\,2$, are linearly independent linear forms in the variables $m_1,\,m_2$; and for all rational numerical values of $m_1$ and $m_2$ such that $(m_1,\,m_2) \neq (0,\,0)$, the  forms $B_j(m_1,\,m_2,\,x_1,\,x_2),\;j=1,\,2$, are linearly independent linear forms in the variables $x_1,\,x_2$;

(ii) on substituting 
\begin{equation}
\begin{aligned}
x_3=&B_1(m_1,\,m_2,\,x_1,\,x_2),\\
x_4=&B_2(m_1,\,m_2,\,x_1,\,x_2), 
\end{aligned}
\label{subs1}
\end{equation}
in Eq.~\eqref{gendeq4var}, we get
\begin{equation}
\psi(x_1,\,x_2)\{\phi_0(m_1,\,m_2)x_1^2+\phi_1(m_1,\,m_2)x_1x_2+\phi_2(m_1,\,m_2)x_2^2\}=0,
\label{deqsubs1}
\end{equation}
where $\psi(x_1,\,x_2)$ is a  binary  form of degree $d-2$ in the variables $x_1,\,x_2$ such that the equation $\psi(x_1,\,x_2)=0 $ has no rational solutions, and $\phi_0(m_1,\,m_2)$,  $\phi_1(m_1,\,m_2),\,\phi_2(m_1,\,m_2)$ are polynomials  that do not vanish simultaneously for any rational values of $m_1$ and $m_2$ such that  $(m_1,\,m_2) \neq (0,\,0)$.

We will show in Section 2.2 that when Eq.~\eqref{gendeq4var} satisfies the above three properties, it is possible  to generate a sequence of rational solutions of Eq.~\eqref{gendeq4var} starting from the known solution.

We note that it is fairly straightforward to determine all those  solutions of Eq.~\eqref{gendeq4var} in which both $x_1,\,x_2$ are 0 or both $x_3,\,x_4$ are 0 since in these cases, Eq.~\eqref{gendeq4var} reduces to a homogeneous equation in just two variables. We will exclude all such solutions of Eq.~\eqref{gendeq4var} since the method described below is based on the substitution given by \eqref{subs1} and if both $x_1,\,x_2$ are 0, we  get the trivial result  $x_3=0,\,x_4=0$, and similarly,  a trivial situation also arises from any solution in which both $x_3,\,x_4$ are 0. Thus, henceforth, whenever we refer to a solution of Eq.~\eqref{gendeq4var} or, a point on the surface defined by Eq.~\eqref{gendeq4var}, we mean a solution (or a point) for which  both $x_1,\,x_2$ are not simultaneously 0 and also  both $x_3,\,x_4$ are not simultaneously 0.

\subsection{}

Given any   rational point  $P=(\alpha_1,\,\alpha_2,\,\alpha_3,\,\alpha_4)$ on a surface of type \eqref{gendeq4var} which satisfies the properties ${\bf D_1}$, ${\bf D_2}$, ${\bf D_3}$,  mentioned in Section 2.1 above, we can, in general,  find two new rational points on the surface  \eqref{gendeq4var}. The first  of these is simply the reflection of $P$, denoted by $R(P)$ and defined by $R(P)=(-\alpha_1,\,\alpha_2,\,\alpha_3,\, \alpha_4)$.  In view of  property ${\bf D_2}$, the point $R(P)$ lies on the surface \eqref{gendeq4var}.  A second rational point on the surface  \eqref{gendeq4var} may be obtained as described below.
 
We note that corresponding to each  solution $(x_1,\,x_2,\,x_3,\,x_4)=(\alpha_1,\,\alpha_2,\,$ $\alpha_3,\, \alpha_4)$ of Eq.~\eqref{gendeq4var}, we can find rational numerical values of the parameters $m_1,\,m_2$ by solving  the two linear relations,

\begin{equation}
\begin{aligned}
\alpha_3=&B_1(m_1,\,m_2,\,\alpha_1,\,\alpha_2),\\
\alpha_4=&B_2(m_1,\,m_2,\,\alpha_1,\,\alpha_2), 
\end{aligned}
\label{relm1m2}
\end{equation}
 obtained by writing $x_i=\alpha_i$ in the relations \eqref{subs1}. We also note that all equivalent solutions $(k\alpha_1,\,k\alpha_2,\,k\alpha_3,\,k\alpha_4)$ yield  the same  pair of values for $m_1$ and $m_2$ which are such that $(m_1,\,m_2) \neq (0,\,0)$. With these values of $m_1,\,m_2$, we make the substitution \eqref{subs1} in Eq.~\eqref{gendeq4var} and reduce it to Eq.~\eqref{deqsubs1}. Since $\psi(x_1,\,x_2) \neq 0$, Eq.~\eqref{deqsubs1} reduces to the following quadratic equation in $x_1$ and $x_2$:
\begin{equation}
\phi_0(m_1,\,m_2)x_1^2+\phi_1(m_1,\,m_2)x_1x_2+\phi_2(m_1,\,m_2)x_2^2=0.
\label{deqsubs1red}
\end{equation}

In view of property ${\bf D_3}$, we note that the coefficients of $x_1^2,\;x_1x_2$ and $x_2^2$ in Eq.~\eqref{deqsubs1red} cannot vanish simultaneously. Since $(x_1,\,x_2,\,x_3,\,x_4)=(\alpha_1,\,\alpha_2,\,\alpha_3,\,$ $\alpha_4)$ is already a known solution of \eqref{gendeq4var}, Eq.~\eqref{deqsubs1red} is necessarily solvable and one of its solutions, corresponding to the known solution, is given by $(x_1,\,x_2)=(\alpha_1,\,\alpha_2)$. At the same time, \eqref{deqsubs1red} has a  second solution  which satisfies the condition $(x_1,\,x_2) \neq (0,\,0)$. With these values of $x_1$ and $x_2$,  we get a new point on the surface Eq.~\eqref{gendeq4var} by using the relations \eqref{subs1}. Since $(m_1,\,m_2) \neq (0,\,0)$, it now follows from condition (i) of property ${\bf D_3}$ that the new point just obtained also satisfies the condition that its last two coordinates are not simultaneously 0. 

Any two solutions of Eq.~\eqref{gendeq4var}  obtained by using the substitution \eqref{subs1} and solving the resulting quadratic equation \eqref{deqsubs1red} will be referred to as conjugate solutions, and the points on the surface \eqref{gendeq4var} corresponding to  conjugate solutions will be referred to as conjugate points. The conjugate of a point $P$ will be denoted by $C(P)$.

Thus, starting from a known rational point $P$  which satisfies  the condition that its first two coordinates are not simultaneously 0 and its last two coordinates are also not simultaneously 0, we can find two new rational points $R(P)$ and $C(P)$ both of which satisfy the condition that their first two coordinates are not simultaneously 0 and their last two coordinates are also not simultaneously 0. 

On the two new rational points just obtained, we  can  again  perform either of the two operations of finding the reflection or the conjugate (written briefly as the $R$ and $C$ operations respectively),   and in this manner,  we can continue to perform these two operations any number of times and in any order. At each step, the  rational points thus generated satisfy the condition that their first two coordinates are not simultaneously 0 and their last two coordinates are also not simultaneously 0, and hence we can actually execute the $R$ and $C$ operations any number of times.

 As long as it is clear in which order the $R$ and $C$ operations are to be performed, we may drop the parentheses --- for instance, we may  write the point $R(C(C(R(C(P)))))$ simply as $RCCRC(P)$. We may similarly also insert parentheses as long as there is no ambiguity.

It readily follows from the definitions of $C(P)$ and $R(P)$ that
\[ C^2(P)=C(C(P))=P \quad \mbox{\rm and}\quad R^2(P)=R(R(P))=P.\]
Therefore new rational  points on the surface \eqref{gendeq4var} can  be generated only if we  perform the operations of finding the reflection and the conjugate of a point alternately. 

The combined operation of first finding the conjugate $C(P)$ of a point $P$, and then finding  the reflection   of the point $C(P)$ will be referred to as the $RC$ operation, and the point thus obtained will be written as $RC(P)$. We can repeat this $RC$ operation on the point $RC(P)$ to obtain a new point, denoted by $(RC)^2(P)$, and we may continue this process any number of times. The point obtained by performing the $RC$ operation $k$ times will be denoted by $(RC)^k(P)$. We thus obtain a sequence of rational points on the surface \eqref{gendeq4var} given by,
\begin{equation} P,\;\;RC(P),\;\;(RC)^2(P),\;\;(RC)^3(P),\ldots,\,\;\;(RC)^k(P),\ldots,\;\;. \label{seq1} \end{equation}

Similarly we define the $CR$ operation as the combined operation of first taking the reflection of  a point $P$  on the surface \eqref{gendeq4var} and then finding its conjugate, and  the point thus obtained is denoted by $CR(P)$. By repeating the $CR$ operation $k$ times, we obtain the point $(CR)^k(P)$, and we thus get the sequence of rational points on the surface \eqref{gendeq4var} given by,
\begin{equation} P,\;\;CR(P),\;\;(CR)^2(P),\;\;(CR)^3(P),\ldots,\,\;\;(CR)^k(P),\ldots,\;\;.\label{seq2} \end{equation}

For convenience, we define $(RC)^0(P) =P$ and $(CR)^0(P)=P$. It follows from the definitions of the operations of finding the reflection and  the conjugate of rational points on the surface \eqref{gendeq4var} that if $(RC)(P)=Q$ then $P=(CR)(Q)$. Further, if $(RC)^k(P)=Q$ then, for any positive integer $h < k$, we have  $(RC)^{k-h}P=(CR)^h(Q)$. Similarly, it follows that if $(CR)^k(P)=Q$ then, for any positive integer $h < k$, we have  $(CR)^{k-h}(P)=(RC)^h(Q)$. We note here that, in general, $CR(P) \neq RC(P)$. In Section 3.2, we will see a specific example illustrating this fact. 

We note that there may exist certain points $P$ on the surface \eqref{gendeq4var} such that the conjugate point  of $P$ coincides with the point $P$. Such points satisfy the condition $C(P)=P$ and will be called self-conjugate points. Similarly, there may exist points $P$ on the surface \eqref{gendeq4var} such that the reflection of $P$ coincides with the point $P$. Such points that remain invariant under reflection satisfy the condition $R(P)=P$ and, when the definition of reflection is clear from the context, we will simply refer to these points as invariant points.

We will denote the individual coordinates of a point $P$ by $x_i(P),\;i=1,\,2,\,3,\,4$. Similarly, we will denote the individual coordinates of the points $R(P),\;C(P)$ and $RC(P)$  by $x_i(RP),\;x_i(CP)$ and $x_i(RCP),\;i=1,\,2,\,3,\,4$ respectively. 

\subsection{}

With respect to both  the sequences of points \eqref{seq1} and \eqref{seq2}, there exists the  possibility that for some positive integer $k$, the $(k+1)^{\rm th}$ point of the sequence is the same as  either  the initial point $P$ or another point that occurs earlier in the sequence.

For the sequence \eqref{seq1}, this will happen if either $(RC)^k(P)=P$ or $(RC)^k(P)=(RC)^j(P)$  where $1 \leq j < k$. In the latter case, it is readily seen that $(RC)^{k-j}(P)=P$. Thus, in either case, there exists an integer $m$ such that $(RC)^m(P)=P$. We then say that $P$ is a point of finite order and the least positive integer $m$ such that $(RC)^m(P)=P$ will be called the order of $P$ with respect to the $RC$ operation. If $P$ is a point of finite order, the sequence \eqref{seq1} will only generate finitely many points on the surface \eqref{gendeq4var}. 

If there is no integer $m$ such that $(RC)^m(P)=P$, we say that $P$ is a point of infinite order with respect to the $RC$ operation and the sequence of points  \eqref{seq1} gives infinitely many  distinct rational points on the surface \eqref{gendeq4var}.

We similarly define the order of a point $P$ with respect to the $CR$ operation as the least positive integer $m$ such that $(CR)^m(P)=P$. We note that when $(RC)^m(P)=P$, then $(CR)^m(P)=(CR)^m(RC)^m(P)=P$, and similarly when $(CR)^m(P)=P$, then $(RC)^m(P)=P$. It follows that the order of any point $P$ with respect to the $RC$ operation  is the same as the order of $P$ with respect to the $CR$ operation. We may thus simply write that the order of a point $P$ is $m$ without referring to the $RC$ or the $CR$ operation. 

\subsection{}

For any specific surface defined by an  equation \eqref{gendeq4var} satisfying the three properties ${\bf D_1,\;D_2,\;D_3}$ mentioned in Section 2.1,  we may obtain a sequence of rational points \eqref{seq1} or \eqref{seq2}, and we may then wish to establish that this is indeed an infinite sequence of distinct rational points on the surface \eqref{gendeq4var}. We could prove this possibly by using the method of induction. The following lemma gives   an alternative method for proving the existence of an infinite sequence of rational, and hence integer, points on the surface \eqref{gendeq4var}.

\begin{lemma}\label{infinitepts} Let  there exist an equation of type \eqref{gendeq4var} satisfying the properties ${\bf D_1,\; D_2,\;D_3}$. If, on the surface defined by Eq.~\eqref{gendeq4var}, the number of rational self-conjugate points be $n_1$ and the number of rational invariant points be $n_2$,   then there are infinitely many  integer points on the  surface \eqref{gendeq4var} if $n_1+n_2$ is an odd integer. 
\end{lemma}

\begin{proof} Starting from  each rational self-conjugate or invariant point $P$,  we will repeatedly apply the $RC$ operation to generate sequences of rational points of the type \eqref{seq1}. We will show that if the initial point $P$ is of finite order, the  total number of self-conjugate points  and invariant  points in each such sequence is exactly 2. We will also show that  any two of the sequences either consist of exactly the same points or they have no points in common. Thus the total  number of self-conjugate points and invariant  points in all such sequences put together is an even number. Since $n_1+n_2$ is odd, there remains at least one self-conjugate point or one invariant point that must necessarily be of infinite order. It follows that there are infinitely many rational, and hence integer,  points on the surface \eqref{gendeq4var}.

We will  now prove  that in each sequence  obtained by starting from a rational self-conjugate or an invariant point of finite order, the  total number of self-conjugate points and invariant  points is exactly 2. If the initial point $P$ is of order 1, then $RC(P)=P$ so that $C(P)=R(P)=P$, and while the sequence \eqref{seq1} reduces to just the single point $P$, this  point  is counted once as a self-conjugate point and once as an invariant point. Thus the total number of self-conjugate points and invariant points in the sequence is to be taken as 2. We also note that such a point $P$ is also included twice in the number $n_1+n_2$, once as a self-conjugate point and once as an invariant point.

Next let $P$ be a rational self-conjugate point of finite order $m \geq 2$ so that $C(P)=P$ and 
\begin{equation}
(RC)^m(P)=P. \label{Porderm}
\end{equation}
Starting with   the point $P$ and repeatedly applying the $RC$ operation, we get a  sequence of $m$ distinct rational points,
\begin{equation} P,\;(RC)(P),\;(RC)^2(P),\;\ldots,\;(RC)^{m-1}(P).\label{seqL1} \end{equation}
We will now determine the points of the sequence \eqref{seqL1} that are self-conjugate or invariant. 

If $(RC)^h(P),\; 1 \leq h \leq m-1$ is a self-conjugate point, then 
\[
\begin{aligned}
 C(RC)^h(P)&=(RC)^h(P),\\
{\rm or,} \quad \quad  (CR)^hC(RC)^h(P)&=(CR)^h(RC)^h(P),\\
{\rm or,} \quad \quad  C(RC)^{2h}(P)&=P,\\
{\rm or,} \quad \quad  CC(RC)^{2h}(P)&=C(P),\\
{\rm or,} \quad \quad  (RC)^{2h}(P)&=P.  
\end{aligned}
\]
We also note that when $m$ is even, it follows from \eqref{Porderm} that $(RC)^{m/2}(P)=(CR)^{m/2}(P)$, and  hence, 
\[C(RC)^{m/2}(P)=C(CR)^{m/2}(P)=CC(RC)^{m/2-1)}RC(P)=(RC)^{m/2}(P).\]

It follows that when $m$ is odd, the sequence \eqref{seqL1} contains exactly one   self-conjugate point, namely the initial  point $P$,  and  when $m$ is even,  it contains exactly two self-conjugate points, namely the point $P$ and the point   $(RC)^{m/2}(P)$,  which is distinct from $P$.

If $(RC)^h(P),\; 1 \leq h \leq m-1$ is an invariant point, then
\[
\begin{aligned}
 R(RC)^h(P)&=(RC)^h(P),\\
{\rm or,} \quad \quad    R(RC)^h(P)&=R(CR)^{h-1}C(P),\\ 
{\rm or,} \quad \quad    R(RC)^h(P)&=R(CR)^{h-1}(P),\quad {\rm since}\;\; C(P)=P,\\
{\rm or,} \quad \quad    (RC)^h(P)&=(CR)^{h-1}(P),\\
{\rm or,} \quad \quad  (RC)^{h-1}(RC)^h(P)&=(RC)^{h-1}(CR)^{h-1}(P),\\
{\rm or,} \quad \quad    (RC)^{2h-1}(P)&=P.
\end{aligned}
\]
We also note that when m is  odd, it follows from \eqref{Porderm} that $(RC)^{(m+1)/2}(P) = (CR)^{(m-1)/2}(P)$, and hence, 
\[R(RC)^{(m+1)/2}(P) = R(CR)^{(m-1)/2}C(P)=(RC)^{(m+1)/2}(P).\] 

It follows that when $m$ is odd, the sequence \eqref{seqL1} contains exactly one invariant point distinct from $P$, namely $(RC)^{(m+1)/2}(P)$, but when $m$ is even,  it does not contain any invariant points.

Thus, whether $m$ is odd or even, the total number of self-conjugate points and invariant points in the sequence of points \eqref{seqL1}, obtained by starting from a rational self-conjugate point of finite order, is  exactly 2 (including the initial self-conjugate point $P$). 

Next, let $P$ be a rational   invariant point of finite order $m \geq 2$, so that $R(P)=P$ and   the relation \eqref{Porderm} is also satisfied. We will now determine the points of the sequence \eqref{seqL1} that are self-conjugate or invariant.

If $(RC)^h(P),\; 1 \leq h \leq m-1$, is a self-conjugate point, then
\[
\begin{aligned}
 C(RC)^h(P)&=(RC)^h(P),\\
\mbox{\rm or,} \quad \quad  (CR)^hC(RC)^h(P)&=(CR)^h(RC)^h(P),\\
\mbox{\rm or,} \quad \quad  C(RC)^{2h}(P)&=P,\\
\mbox{\rm or,} \quad \quad  RC(RC)^{2h}(P)&=R(P),\\
\mbox{\rm or,} \quad \quad  (RC)^{2h+1}(P)&=P, \quad \mbox{\rm since}\;\; R(P)=P. 
\end{aligned}
\]
We also note that when $m$ is odd, it follows from \eqref{Porderm} that $(RC)^{(m-1)/2}(P)=(CR)^{(m+1)/2}(P)$, and hence,
\[C(RC)^{(m-1)/2}(P)=C(CR)^{(m+1)/2}R(P)=(RC)^{(m-1)/2}(P).\]

It follows that when $m$ is is odd, the sequence \eqref{seqL1} contains exactly one  self-conjugate point, namely    $(RC)^{(m-1)/2}(P)$,  which is distinct from $P$ but when $m$ is even, it  does not contain any  self-conjugate points.

 If $(RC)^h(P),\;1 \leq h \leq m-1$ is a point that remains invariant under reflection, then 
\[
\begin{aligned}
 R(RC)^h(P)&=(RC)^h(P),\\
{\rm or,} \quad \quad   R(RC)^h(P)&=(RC)^hR(P),\quad {\rm since} R(P)=P,\\ 
{\rm or,} \quad \quad   R(RC)^h(P)&=R(CR)^{h}(P),\\
{\rm or,} \quad \quad    (RC)^h(P)&=(CR)^{h}(P),\\
{\rm or,} \quad \quad  (RC)^h(RC)^h(P)&=(RC)^h(CR)^{h}(P),\\
{\rm or,} \quad \quad    (RC)^{2h}(P)&=P.
\end{aligned}
\]
We also note that when $m$ is even, it follows from \eqref{Porderm} that $(RC)^{m/2}P=(CR)^{m/2}(P)$, and hence,
\[R(RC)^{m/2}P=R(CR)^{m/2}(P)=(RC)^{m/2}R(P)=(RC)^{m/2}(P).\]

It follows that when $m$ is odd, the sequence \eqref{seqL1} contains exactly one  invariant point, namely the initial point $P$, and  when $m$ is even, it  contains exactly two  invariant points, namely the point $P$ and the point $(RC)^{m/2}P$ which is distinct from $P$.

Thus, whether $m$ is odd or even, the total number of self-conjugate points and invariant points in the sequence of points \eqref{seqL1},  obtained by starting from a rational  invariant point of finite order,  is also  exactly 2 (including the initial invariant point $P$).

We have thus shown that in all the sequences of type \eqref{seqL1}, obtained by starting from  a rational self-conjugate or an invariant point of finite order, the total number of self-conjugate points and invariant points  is exactly 2.  We will now show that two such  sequences either consist of exactly the same points or they have no points in common.

If any two sequences of the type \eqref{seqL1}, obtained  by starting from two distinct rational points $P$ and $P_1$, of finite orders $m$ and $m_1$ respectively, have a point in common, we must have $(RC)^h(P)=(RC)^{h_1}(P_1)$ for some integers $h$ and $h_1$. There exists an integer $\lambda$ such that $\lambda m+h \geq h_1 $  and it follows from the relation $(RC)^h(P)=(RC)^{h_1}(P_1)$ that $(RC)^{\lambda m+h-h_1}(P)=P_1$. It now follows that all points of the second sequence are included in the first sequence. Similarly, there exists an  integer $\lambda_1$ such that $\lambda_1 m_1+h_1 \geq h$ and it follows from the relation $(RC)^h(P)=(RC)^{h_1}(P_1)$ that $(RC)^{\lambda_1 m_1+h_1-h}(P_1)=P$. It now follows that  all points of the first sequence are included in the second sequence, and hence the two sequences consist of exactly the same points. 

Thus the total  number of self-conjugate points and invariant points in all the distinct sequences of type \eqref{seqL1}, obtained by starting from  all the rational self-conjugate or  invariant points of finite order,  is an even integer.    Since the total number of  rational self-conjugate points and invariant points is an odd integer, at least one of these points cannot be of finite order. Starting from such a point that is not of finite order, we obtain an infinite number of rational, and hence integer, points on the surface \eqref{gendeq4var}. This proves the lemma. 
\end{proof}

\subsection{}
We have, for the sake of simplicity,  stipulated in property ${\bf D_2}$ that the form $f(x_i)$ should satisfy the condition $f(x_1,\,x_2,\,x_3,\,x_4)$ $=f(-x_1,\,x_2,\,x_3,\,x_4)$. In fact, our method will work as long as the form  $f(x_i)$ possesses a certain symmetry so that when we know one solution of Eq.~\eqref{gendeq4var}, we immediately get another solution using the symmetry. For instance, the form  $f(x_i)$ may satisfy the condition $f(x_1,\,x_2,\,x_3,\,x_4)=f(x_2,\,x_1,\,x_3,\,x_4)$ or $f(x_1,\,x_2,\,x_3,\,x_4)=f(x_3,\,x_4,\,x_1,\,x_2)$. A more interesting diophantine equation is
\begin{equation}
F(x_1,\,x_2)+4F(x_3,\,x_4)=0, \label{inteqn}
\end{equation}
where $F(x_1,\,x_2)$ is a binary quartic form, and we note that whenever a point $P$ given by $(\alpha_1,\,\alpha_2,\,\alpha_3,\,\alpha_4)$ lies on the surface \eqref{inteqn}, another point on the surface  \eqref{inteqn} is given by $(2\alpha_3,\,2\alpha_4,\,\alpha_1,\,\alpha_2)$. Thus, in this case, we define $R(P)=(2\alpha_3,\,2\alpha_4,\,\alpha_1,\,\alpha_2)$.  If Eq.~\eqref{inteqn} also satisfies the properties  ${\bf D_1}$ and  ${\bf D_3}$, we can obtain a sequence of points on the surface \eqref{inteqn} by the method described above.

\subsection{} In Sections 3, 4 and 5, we will show that there exist    infinitely many  examples of quartic and higher degree diophantine equations of type \eqref{gendeq4var} which have infinitely many solutions. These infinitely many solutions are obtained by reducing \eqref{gendeq4var} to an equation of type \eqref{deqsubs1red}. Now Eq.~\eqref{deqsubs1red} is a quadratic equation in the variables $x_1$ and $x_2$, and it has a rational solution if and only if its discriminant $\phi_1^2(m_1,\,m_2)-4\phi_0(m_1,\,m_2)\phi_2(m_1,\,m_2)$ becomes a perfect square. It readily follows that when  Eq.~\eqref{gendeq4var} has infinitely many solutions, the nonhomogeneous polynomial diophantine equation in the variables $m_1,\,m_2$ and $z$ given by
\begin{equation}
\phi_1^2(m_1,\,m_2)-4\phi_0(m_1,\,m_2)\phi_2(m_1,\,m_2)=z^2, \label{deqm1m2}
\end{equation}
also simultaneously has infinitely many solutions in rational numbers.

\section{Quartic diophantine equations}\label{quarteqn}
In Section 3.1 we will  construct a very general quartic equation that can be solved by the method described in Section 2. In Sections 3.2 and 3.3, we  discuss  some specific quartic diophantine equations  and prove that there exist infinitely many quartic diophantine equations  for which we can obtain  an arbitrarily large number of solutions by applying this method.

\subsection{A general quartic equation} 
Consider the  quadratic forms in the variables $x_1,\,x_2,\,x_3,\,x_4,$ given by,
\begin{equation}
\begin{aligned}
Q(x_1,\,x_2)&=x_1^2+px_1x_2+qx_2^2,\\
Q_1(x_1,\,x_2,\,x_3,\,x_4)&=x_1x_3-qx_2x_4,\\
Q_2(x_1,\,x_2,\,x_3,\,x_4)&=x_1x_4+x_2x_3+px_2x_4,
 \end{aligned}
\label{Mdefqd}
\end{equation}
where $p,\,q$ are arbitrary parameters. If we write 
\begin{equation}
x_3=m_1x_1+(pm_1+qm_2)x_2, \quad x_4=m_2x_1-m_1x_2, \label{Msubs1}
\end{equation}
where $m_1$ and $m_2$ are arbitrary parameters, then we have the identities,
\begin{equation}
\begin{aligned}
Q(x_3,\,x_4)&=Q(m_1,\,m_2)Q(x_1,\,x_2),\\
Q_1(x_1,\,x_2,\,x_3,\,x_4)&=m_1Q(x_1,\,x_2),\\
Q_2(x_1,\,x_2,\,x_3,\,x_4)&=m_2Q(x_1,\,x_2). 
\end{aligned} 
\label{Mrelqd}
\end{equation}

We  now consider the quartic equation,
\begin{equation}
F(x_1,\,x_2,\,x_3,\,x_4)=0, \label{quarteqngen}
\end{equation}
where  $F(x_1,\,x_2,\,x_3,\,x_4)$ is the quartic form defined by,
\begin{multline}
F(x_1,\,x_2,\,x_3,\,x_4)=a_1Q(x_1,\,x_2)Q(-x_1,\,x_2)\\
+a_2Q_1(x_i)Q_1(-x_1,\,x_2,\,x_3,\,x_4)+a_3Q_2(x_i)Q_2(-x_1,\,x_2,\,x_3,\,x_4)\\
+a_4\{Q(x_1,\,x_2)Q_1(-x_1,\,x_2,\,x_3,\,x_4)+Q(-x_1,\,x_2)Q_1(x_i)\}\\
+a_5\{Q(x_1,\,x_2)Q_2(-x_1,\,x_2,\,x_3,\,x_4)+Q(-x_1,\,x_2)Q_2(x_i)\}\\
+(a_6x_1^2+a_7x_2^2+a_8x_2x_3+a_9x_2x_4+a_{10}x_3^2+a_{11}x_3x_4+a_{12}x_4^2)Q(x_3,\,x_4), \label{quartgen}
\end{multline}
with $a_j,\;j=1,\,2,\,\ldots,\,12$, being  arbitrary integer coefficients.

 It is readily seen  that $F(x_1,\,x_2,\,x_3,\,x_4)=F(-x_1,\,x_2,\,x_3,\,x_4)$. Thus, the form $F(x_1,\,x_2,\,x_3,\,x_4)$ satisfies the property ${\bf D_2}$. Further, if in the form $F(x_i)$, we substitute the values of $x_3,\,x_4$ given by \eqref{Msubs1},  it is readily observed, in view of the relations \eqref{Mrelqd}, that Eq.~\eqref{quarteqngen} reduces to an equation of type \eqref{deqsubs1} where $\psi(x_1,\,x_2)$ is, in fact, given by $Q(x_1,\,x_2)$. 

The parameters $p,\,q$ and the coefficients $a_j$ can easily be chosen such that   the property ${\bf D_1}$ and the remaining conditions stipulated in property ${\bf D_3}$ are  satisfied.

Thus, we can find a sequence of rational solutions of Eq.~\eqref{quarteqngen} starting from a known solution. In the next two subsections, we consider a couple of special cases of Eq.~\eqref{quarteqngen}.

\subsection{}
In Eq.~\eqref{quarteqngen}, we take \[
\begin{aligned}
&p&=&3,&\;\; &q&=&-1,&\;\; &a_1&=&1,&\;\;&a_2&=&-2h,&\;\;&a_3&=&-h,&\\
&a_4&=&-h,&\;\;&a_5&=&-h,&\;\;&a_6&=&-h,&\;\;&a_7&=&4h,&\;\;&a_8&=&-h,&\\
&a_9&=&-h,&\;\;&a_{10}&=&1,&\;\; &a_{11}&=&3,&\;\;&a_{12}&=&-1,&& &&
\end{aligned}\] 
where $h$ is an arbitrary integer parameter, and  we thus get the following equation:
\begin{multline}
x_1^4-11x_1^2x_2^2+x_2^4+x_3^4+6x_3^3x_4+7x_3^2x_4^2-6x_3x_4^3+x_4^4\\
+h(4x_1^2x_2x_3-2x_1^2x_2x_4+x_1^2x_3^2-3x_1^2x_3x_4+2x_1^2x_4^2\\
+2x_2^3x_3+8x_2^3x_4+3x_2^2x_3^2+6x_2^2x_3x_4-15x_2^2x_4^2\\
-x_2x_3^3-4x_2x_3^2x_4-2x_2x_3x_4^2+x_2x_4^3)=0.
\label{quarteqnex1}
\end{multline}

It is easily verified that for all values of $h$,  a solution of Eq.~\eqref{quarteqnex1} is given by $(x_1,\,x_2,\,x_3,\,x_4)=(1,\,1,\,1,\,1)$. 

On substituting the values of $x_3,\;x_4$ given by 
\begin{equation}
x_3 = m_1x_1+(3m_1-m_2)x_2,\quad x_4 = m_2x_1-m_1x_2, \label{subsquarteqnex1}
\end{equation}
Eq.~\eqref{quarteqnex1} reduces, after removing the irreducible factor $x_1^2+3x_1x_2-x_2^2$, to
 the following equation: 
\begin{multline}
\{(m_1^2-3m_1m_2+2m_2^2)h+m_1^4+6m_1^3m_2+7m_1^2m_2^2-6m_1m_2^3+m_2^4+1\}x_1^2\\
-\{(m_1^3+4m_1^2m_2+2m_1m_2^2-m_2^3-6m_1^2+6m_1m_2+3m_2^2-4m_1+2m_2)h\\
-3m_1^4-18m_1^3m_2-21m_1^2m_2^2+18m_1m_2^3-3m_2^4+3\}x_1x_2\\
-\{(2m_1^3+5m_1^2m_2-5m_1m_2^2+m_2^3-6m_1^2-12m_1m_2+3m_2^2-2m_1-2m_2)h\\
+m_1^4+6m_1^3m_2+7m_1^2m_2^2-6m_1m_2^3+m_2^4+1\}x_2^2=0. \label{quarteqnex1subs1}
\end{multline}

 It is now readily verified that, for any arbitrary value of $h$, Eq.~\eqref{quarteqnex1} satisfies all the three properties ${\bf D_1}$, ${\bf D_2}$ and ${\bf D_3}$ mentioned in Section 2.1.  

Now,  starting from the point $P_0=(1,\,1,\,1,\,1)$ and  repeatedly applying the $RC$ operation, we   obtain a sequence of points $P_0,\,P_1,\,P_2,\,\dots,\,P_n$ on the surface \eqref{quarteqnex1}. The coordinates of these points are rational functions of $h$. We will show that given an arbitrary positive integer $n$, howsoever large, there exists an  integer value of $h$ such that these   $n$  points are distinct and rational. 

If the points $P_0,\,P_1,\,P_2,\,\dots,\,P_n$ are not distinct,  $P_0$ must be a point of finite order not exceeding $n$. The point $P_0$ will be  of finite order $r < n$ only if the following conditions are satisfied:
\begin{equation}
x_1(P_r)=x_2(P_r)=x_3(P_r)=x_4(P_r) \neq 0, \label{condquarteqnex1}
\end{equation}
where, as already noted, the coordinates $x_i(P_r),\; i=1,\,2,\,3,\,4$, are rational functions of $h$.  We now show that these conditions cannot be satisfied identically for all values of $h$. 

We consider the case $h=0$ when Eq.~\eqref{quarteqnex1} may be written as,
\begin{equation}
x_1^4-11x_1^2x_2^2+x_2^4+x_3^4+6x_3^3x_4+7x_3^2x_4^2\\
-6x_3x_4^3+x_4^4=0.
\label{quarteqnex1spl}
\end{equation}
while Eq.~\eqref{quarteqnex1subs1} reduces  to
 the following equation: 
\begin{multline}
(m_1^4+6m_1^3m_2+7m_1^2m_2^2-6m_1m_2^3+m_2^4+1)x_1^2\\
+3(m_1^2+3m_1m_2-m_2^2-1)(m_1^2+3m_1m_2-m_2^2+1)x_1x_2\\
-(m_1^4+6m_1^3m_2+7m_1^2m_2^2-6m_1m_2^3+m_2^4+1)x_2^2=0.
\label{quarteqnex1splsubs1}
\end{multline}

Let a rational point $P$ on the surface \eqref{quarteqnex1spl} be given by $(3^t,\,3^t,\,\alpha,\,\beta)$ where $\alpha,\,\beta$ are integers such that $\alpha \equiv 1\; (\bmod \;3)$ , $\beta \equiv 1\; (\bmod \;3)$, and $t$ is  a nonnegative integer. The values of $m_1,\,m_2$ corresponding to the point $P$ are given by 
\begin{equation}
m_1= (\alpha+\beta)/3^{t+1},\;\; m_2= (\alpha+4\beta)/3^{t+1}.
\end{equation}

With these values of $m_1, \,m_2$,  we now work out the conjugate $C(P)$ of the point $P$. Eq.~\eqref{quarteqnex1splsubs1} gives the values of the ratios $x_1/x_2$ for a point $P$ and its conjugate $C(P)$.  It follows from \eqref{quarteqnex1splsubs1} that
\[
\frac{x_1(P)}{x_2(P)}.\frac{x_1(CP)}{x_2(CP)} = -1.
\]
Since $x_1(P)=x_2(P)=3^k$, we get $-x_1(CP)=x_2(CP)=m$ where $m$ is an arbitrary rational number, and we obtain the values of  $x_3(CP),\,x_4(CP)$ using the relations \eqref{subsquarteqnex1}. We thus obtain the point $C(P)$, and on taking the reflection of this point, we  obtain the point $RC(P)$ which is given by $(3^{t+1},\,3^{t+1},\,\alpha^{\prime},\,\beta^{\prime})$ where $\alpha^{\prime} =2(\alpha+\beta)-3\alpha,\,\beta^{\prime}=2(\alpha+\beta)+3\beta$. Thus, we get $\alpha^{\prime} \equiv 1\; (\bmod \; 3)$, $\beta^{\prime} \equiv \; (\bmod \; 3)$ and it now follows by induction that the sequence of points on the surface  \eqref{quarteqnex1spl},  obtained by starting from $(1,\,1,\,1,\,1)$ as the initial point $P_0$, consists of infinitely many distinct rational points. The first four points of the sequence obtained in this manner are $P_0=(1,\,1,\,1,\,1)$, $P_1=(3,\,3,\,1,\,7)$, $P_2=(3^2,\,3^2,\,13,\,37)$, $P_3=(3^3,\,3^3,\,61,\,211)$.

Reverting to the original equation \eqref{quarteqnex1}, we can now  conclude that the conditions \eqref{condquarteqnex1} cannot be identically satisfied for all values of $h$ since they are not satisfied in the special case $h=0$. The conditions \eqref{condquarteqnex1} are, in fact,  polynomial equations in $h$ and they can be satisfied by, at most,  a finite number of integer values of $h$. Excluding these values of $h$, we can still assign infinitely many integer values to $h$, and thus obtain infinitely many surfaces \eqref{quarteqnex1} on which there are  at least $n$ distinct rational points where $n$ is any arbitrarily chosen positive integer.

When $h=7$, the first four points of the sequence of points found on the surface \eqref{quarteqnex1} starting from the point $P_0$ are as follows:

\[(1,\,1,\,1,\,1),\quad (-130,\,  31, \, 97,\, 196), \]
\[(2244925401,\,1768375579,\, -3244635281, \,5477857719),\]
and 
\[\begin{aligned}(&48174715330795614136594596008400681550325,\\
&39780816364918032832888229561551900194758,\\
& -90144829539953389017242502188382745170382, \\&-131064449157792994852020542549892912835855).
\end{aligned}
\]
Further, while $RC(P_0)=(-130,\,  31, \, 97,\, 196)$, we note that $CR(P_0)=(15,\,8,\,$ $ 8,\,-15)$. This illustrates the fact that, in general,  $RC(P) \neq CR(P)$  as was stated in Section 2.2. 

We have obtained  infinitely many  integer solutions of Eq.~\eqref{quarteqnex1spl} of the type $(3^k,\,3^k,\,\alpha,\,\beta)$. These solutions satisfy the condition $x_1=x_2$. We note that the complete solution of \eqref{quarteqnex1spl} satisfying the additional  condition $x_1$ $=x_2$ is readily obtained and is given by
\begin{equation}
x_1 = x_2= r^2-8rs+3s^2,\;\; x_3 = r^2-6rs+21s^2,\;\; x_4 = 2r^2-6s^2,
\end{equation}
and
\begin{equation}
x_1 = x_2= -r^2+rs+3s^2,\;\; x_3 = -r^2+6rs-6s^2, \;\;x_4 = r^2+3s^2,
\end{equation}
where $r$ and $s$ are arbitrary parameters. Thus Eq.~\eqref{quarteqnex1spl} is, by itself, not of much intrinsic interest except that the infinite sequence of points $P_0,\,P_1,\,P_2,$ $\dots,$ found on the surface \eqref{quarteqnex1spl} has been used to prove that there exist infinitely many values of $h$ for which the diophantine equation \eqref{quarteqnex1} has an arbitrarily large number of solutions. 

We note that for any arbitrary value of $h$, the surface defined by Eq.~\eqref{quarteqnex1} has singularities at the points $(0,\,0,\, -3+\sqrt{13},\,2)$ and $(0,\,0,\,-3-\sqrt{13},\,2)$.

We give below the nonhomogeneous equation obtained by equating the discriminant of Eq.~\eqref{quarteqnex1subs1} to a perfect square when $h=7$:
\begin{multline}
13m_1^8+156m_1^7m_2+650m_1^6m_2^2+936m_1^5m_2^3-273m_1^4m_2^4-936m_1^3m_2^5\\
+650m_1^2m_2^6-156m_1m_2^7+13m_2^8+14m_1^7+56m_1^6m_2-294m_1^5m_2^2\\
-1554m_1^4m_2^3-966m_1^3m_2^4+1806m_1^2m_2^5-644m_1m_2^6+70m_2^7\\
+161m_1^6+392m_1^5m_2-2254m_1^4m_2^2-4606m_1^3m_2^3+4046m_1^2m_2^4\\
-952m_1m_2^5+63m_2^6-84m_1^5-1428m_1^4m_2-1722m_1^3m_2^2+6384m_1^2m_2^3\\
-1596m_1m_2^4-42m_2^5+186m_1^4-3784m_1^3m_2+5222m_1^2m_2^2-3860m_1m_2^3\\
+1411m_2^4+2058m_1^3-2436m_1^2m_2+336m_1m_2^2-210m_2^3+392m_1^2\\
-952m_1m_2+462m_2^2-224m_1+28m_2+13 =z^2. \label{eqndiscex1}
\end{multline}

We have given above the first four solutions of Eq.~\eqref{quarteqnex1} when $h=7$, obtained by starting from the point $P_0$, and  corresponding to these solutions of Eq.~\eqref{quarteqnex1}, we can readily find rational solutions of Eq.~\eqref{eqndiscex1}. The values of $(m_1,\,m_2,\,z)$ for the first three  solutions of Eq.~\eqref{eqndiscex1} are as follows:
\[
\begin{array}{c}
(2/3,\, 5/3,\,66),\quad (-2178/1283,\, -1415/1283,\,249877066/1646089),\\[0.2in]
({\dfrac{1872911657224140}{10773306028890859}},\;{\dfrac{ 27763349640930281}{ 10773306028890859}},\\[0.2in]
{\dfrac{ 32394133198932561746630798860799374}{ 116064122792136130054389733757881}}).
\end{array}
\]
The next solution involves integers consisting of more than 64 digits and is hence omitted.

\subsection{}
As a second example, in Eq.~\eqref{quarteqngen} we take, 
\[
\begin{aligned}
&p&=&1,&\;\; &q&=&3,&\;\;&a_1&=&6,&\;\;&a_2&=&6,&\;\;&a_3&=&0,&\\
&a_4&=&1,&\;\;&a_5&=&6,&\;\;&a_6&=&-3,&\;\;&a_7&=&-105,&&a_8&=&12,&\\
&a_9&=&6,&\;\;&a_{10}&=&3,&\;\; &a_{11}&=&3,&\;\;&a_{12}&=&-3,&&&&
\end{aligned}\] 
when we get the quartic equation, 
\begin{multline}
6x_1^4+(30x_2^2+10x_2x_3-6x_2x_4-9x_3^2-3x_3x_4-9x_4^2)x_1^2\\
+54x_2^4+36x_2^3x_3+18x_2^3x_4-105x_2^2x_3^2-105x_2^2x_3x_4\\
-261x_2^2x_4^2+12x_2x_3^3+18x_2x_3^2x_4+42x_2x_3x_4^2+18x_2x_4^3\\
+3x_3^4+6x_3^3x_4+9x_3^2x_4^2+6x_3x_4^3-9x_4^4=0. \label{quarteqnex2}
\end{multline}

It is easily verified that $(x_1,\,x_2,\,x_3,\,x_4)=(0,\,1,\,1,\,0)$ is a solution of Eq.~\eqref{quarteqnex2}.

On substituting 
\begin{equation}
x_3 = m_1x_1+(m_1+3m_2)x_2,\quad  x_4 =m_2x_1-m_1x_2,
\end{equation}
Eq.~\eqref{quarteqnex2} reduces, after removing factor $(x_1^2+x_1x_2+3x_2^2)$, to
\begin{multline}
(3m_1^4+6m_1^3m_2+9m_1^2m_2^2+6m_1m_2^3-9m_2^4-9m_1^2-3m_1m_2-9m_2^2+6)x_1^2\\
+(3m_1^4+30m_2m_1^3+45m_1^2m_2^2+90m_1m_2^3+27m_2^4+12m_1^3+18m_1^2m_2\\
+42m_1m_2^2+18m_2^3-6m_1^2-36m_1m_2+10m_1-6m_2-6)x_1x_2-(3m_1^4-6m_1^3m_2\\
-27m_1^2m_2^2-54m_1m_2^3-81m_2^4-6m_1^3-42m_1^2m_2-54m_1m_2^2-108m_2^3\\
+87m_1^2+105m_1m_2+315m_2^2-6m_1-36m_2-18)x_2^2=0. \label{quarteqnex2subs1}
\end{multline}

It is now readily verified that Eq.~\eqref{quarteqnex2} satisfies all the three properties ${\bf D_1}$, ${\bf D_2}$ and ${\bf D_3}$ mentioned in Section 2.1.  Thus we may generate  a sequence of rational points on the surface \eqref{quarteqnex2} taking $(0,\,1,\,1,\,0)$ as the initial point. 

We will now show that the total number of all  rational self-conjugate points and rational  invariant points on the surface \eqref{quarteqnex2} is an odd integer, and then apply Lemma~\ref{infinitepts} to prove that there are infinitely many  integer solutions of Eq.~\eqref{quarteqnex2}.

The discriminant $d(m_1,\,m_2)$ of Eq.~\eqref{quarteqnex2subs1} is given by
\begin{multline}
d(m_1,\,m_2)=45m_1^8+180m_1^7m_2+810m_1^6m_2^2+1800m_1^5m_2^3+4095m_1^4m_2^4\\
+5400m_1^3m_2^5+7290m_1^2m_2^6+4860m_1m_2^7+3645m_2^8+180m_1^6m_2\\
+540m_1^5m_2^2+2160m_1^4m_2^3+3420m_1^3m_2^4+6480m_1^2m_2^5+4860m_1m_2^6\\
+4860m_2^7+1044m_1^6+3384m_1^5m_2+9000m_1^4m_2^2+13536m_1^3m_2^3\\
+10872m_1^2m_2^4+6264m_1m_2^5-8100m_2^6+60m_1^5+492m_1^4m_2+324m_1^3m_2^2\\
+2628m_1^2m_2^3+756m_1m_2^4+4860m_2^5-3036m_1^4-5112m_1^3m_2-15648m_1^2m_2^2\\
-10512m_1m_2^3-13176m_2^4-192m_1^3-504m_1^2m_2-720m_1m_2^2-1512m_2^3\\
+2908m_1^2+3048m_1m_2+8244m_2^2-264m_1-792m_2-396. \label{disquarteqnex2}
\end{multline}

If there exist any  rational self-conjugate points on the surface \eqref{quarteqnex2}, there must exist rational numbers $m_1$ and $m_2$ such that Eq.~\eqref{quarteqnex2subs1} has two coincident roots and hence $d(m_1,\,m_2)$ must be 0. As $m_1$ and $m_2$ are rational numbers, we write $m_1=n_1/n_0,\;m_2=n_2/n_0$ where $n_1,\,n_2$ and $n_0$ are integers such that $n_0 \neq 0$ and ${\rm gcd}(n_0,\,n_1,\,n_2)=1$, and now, on equating $d(m_1,\,m_2)$ to 0, we get an equation that may be written as follows:
\begin{multline}
396n_0^8+(264n_1+792n_2)n_0^7-(2908n_1^2+3048n_1n_2+8244n_2^2)n_0^6\\
+(192n_1^3+504n_1^2n_2+720n_1n_2^2+1512n_2^3)n_0^5\\
+(3036n_1^4+5112n_1^3n_2+15648n_1^2n_2^2+10512n_1n_2^3+13176n_2^4)n_0^4\\
-12(n_1^2+n_1n_2+3n_2^2)(5n_1^3+36n_1^2n_2-24n_1n_2^2+135n_2^3)n_0^3\\
-36(29n_1^2+36n_1n_2-25n_2^2)(n_1^2+n_1n_2+3n_2^2)^2n_0^2\\
-180(n_1^2+n_1n_2+3n_2^2)^3n_0n_2-45(n_1^2+n_1n_2+3n_2^2)^4=0.
\label{disquarteqnex2var1}
\end{multline}

We note that in Eq.~\eqref{disquarteqnex2var1}, all terms except the last are even integers, hence $(n_1^2+n_1n_2+3n_2^2)$ is an even integer and it readily follows that both $n_1$ and $n_2$ must be even integers. On substituting $n_1=2n_3$ and $n_2=2n_4$ in Eq.~\eqref{disquarteqnex2var1}, we get the following equation:
\begin{multline}
99n_0^8+(132n_3+396n_4)n_0^7-(2908n_3^2+3048n_3n_4+8244n_4^2)n_0^6\\
+(384n_3^3+1008n_3^2n_4+1440n_3n_4^2+3024n_4^3)n_0^5\\
+(12144n_3^4+20448n_3^3n_4+62592n_3^2n_4^2+42048n_3n_4^3+52704n_4^4)n_0^4\\
-96(n_3^2+n_3n_4+3n_4^2)(5n_3^3+36n_3^2n_4-24n_3n_4^2+135n_4^3)n_0^3\\
-576(29n_3^2+36n_3n_4-25n_4^2)(n_3^2+n_3n_4+3n_4^2)^2n_0^2\\
-5760n_4(n_3^2+n_3n_4+3n_4^2)^3n_0-2880(n_3^2+n_3n_4+3n_4^2)^4=0,
\label{disquarteqnex2var2}
\end{multline}

We note that in Eq.~\eqref{disquarteqnex2var2}, all terms except the first are even integers, hence $n_0$ must  an even integer. This is a contradiction since ${\rm gcd}(n_0,\,n_1,\,n_2)=1$. Hence the equation $d(m_1,\,m_2)=0$  has no rational solutions. It follows that there are no rational self-conjugate points on the surface \eqref{quarteqnex2}.

Next we determine the number of rational invariant points on the surface \eqref{quarteqnex2}. These points are precisely the rational points on \eqref{quarteqnex2} satisfying the conditions $x_1=0$ and $x_2 \neq 0$. On substituting $x_1=0$ in \eqref{quarteqnex2}, we get the condition,
\begin{multline}
 18x_2^4+12x_2^3x_3+6x_2^3x_4-35x_2^2x_3^2-35x_2^2x_3x_4\\
-87x_2^2x_4^2+4x_2x_3^3+6x_2x_3^2x_4+14x_2x_3x_4^2+6x_2x_4^3 \quad \quad \quad \quad \quad \quad \;\;\;\;\\
+x_3^4+2x_3^3x_4+3x_3^2x_4^2+2x_3x_4^3-3x_4^4=0. \quad \quad \quad \quad  \label{condRP}
\end{multline}

Eq.~\eqref{condRP} represents a curve of genus 3 in projective space and hence it has a finite number of integer solutions. It is readily verified that if $(x_2,\,x_3,\,x_4)$ $=(\alpha_2,\,\alpha_3,\,\alpha_4)$ is any integer  solution of \eqref{condRP} with $\alpha_2 \neq 0$ and $\alpha_4 \neq 0$, then another distinct integer solution of \eqref{condRP} is given by $(x_2,\,x_3,\,x_4)=(\alpha_2,\,\alpha_3+\alpha_4,\,-\alpha_4)$. Thus, integer solutions of \eqref{condRP} can be paired off except for those solutions in which $x_4=0$. On substituting $x_4=0$ in \eqref{condRP}, we get 
\begin{equation}
(x_2-x_3)(18x_2^3+30x_2^2x_3-5x_2x_3^2-x_3^3)=0, \label{condRP1}
\end{equation}
and so we get just one rational solution of \eqref{condRP} with $x_4=0$, namely $( x_2,\,x_3,\,x_4)=(1,\,1,\,0)$. This yields the   single rational invariant point $(0,\,1,\,$ $1,\,0)$.  As the other invariant points occur in pairs, this shows that there are an odd   number of  rational invariant points on the surface \eqref{quarteqnex2}.

As there are no self-conjugate points on the surface \eqref{quarteqnex2}, it follows that the total number of self-conjugate points and invariant  points is an  odd integer. It now follows from Lemma~\ref{infinitepts} that there are infinitely many rational points on the surface \eqref{quarteqnex2}.

We now obtain a sequence of rational points on the surface \eqref{quarteqnex2} by repeatedly applying the $RC$ operation starting from the point $(0,\,1,\,1,\,0)$. The first four points of the sequence are 
\[(0,\,1,\,1,\,0),\quad  (-63,\,44,\,44,\, 21),\]
\[(-53966863944,\, 3060077833,\, -43118821745,\,14262975216),\]
and
\[
\begin{aligned}
(&-147609072097506717422185174080259362584795541167553561,\\
&40226536355947814403673938708520812105643977659529244, \\
&111318341489401190419048903211940940726579266647898028,\\
&-77464352589959936491092062638791236314929817761076995).
\end{aligned}
 \]

Each solution of the above sequence yields a solution of  the nonhomogeneous diophantine equation,
\begin{equation}
d(m_1,\,m_2)=z^2, \label{discrimsqquarteqnex2}
\end{equation}
where $d(m_1,\,m_2)$ is given by \eqref{disquarteqnex2}. The values of $(m_1,\,m_2,\,z)$ for the first three solutions of Eq.~\eqref{discrimsqquarteqnex2}, corresponding to the respective solutions of Eq.~\eqref{quarteqnex2}, are given by 
\[
\begin{array}{c}
(0, \,1/3,\,7), \quad (-1848/2335,\, 1537/7005,\,122166401/5452225),\\[0.2in]
(\dfrac{ 732016714554891233832}{ 925123942970233269817},\, -\dfrac{ 858029179234358903561}{2775371828910699809451},\\[0.2in]
\dfrac{ 14804259873811552782039788426057859064445551}{ 855854309856791419406657915087787523213489}).
\end{array}
\]

The surface defined by  the quartic equation \eqref{quarteqnex2}  has no singularities.  It is a K3 surface whose  arithmetic genus and  geometric genus are both 1. These computations were done using the online MAGMA calculator \cite{Mg}.

\section{Sextic diophantine equations}\label{sexticeqn}

As in the case of quartic equations, we will first construct, in Section 4.1, a very general sextic equation that can be solved by the method described in Section 2. In Sections 4.2 and 4.3, we will discuss specific numerical examples of sextic equations.
 
\subsection{A general sextic equation}

We first define one quadratic and two cubic forms in the variables $x_1,\,x_2,\,x_3,$ $x_4,$ as follows:
\begin{equation}
\begin{aligned}
Q(x_1,\,x_2)&=x_1^2+px_1x_2+qx_2^2,\\
C_1(x_1,\,x_2,\,x_3,\,x_4)&=x_1^2x_3+qx_2^2x_3+pqx_2^2x_4,\\
C_2(x_1,\,x_2,\,x_3,\,x_4)&=x_1^2x_4-px_2^2x_3-(p^2-q)x_2^2x_4,
 \end{aligned}
\label{sextfms}
\end{equation}
where $p,\,q$ are arbitrary parameters. If we write 
\begin{equation}
x_3=m_1x_1+(pm_1+qm_2)x_2, \quad x_4=m_2x_1-m_1x_2, \label{sextsubs1}
\end{equation}
where $m_1$ and $m_2$ are arbitrary parameters,  we have the following identities:
\begin{equation}
\begin{aligned}
Q(x_3,\,x_4)&=Q(m_1,\,m_2)Q(x_1,\,x_2),\\
C_1(x_1,\,x_2,\,x_3,\,x_4)&=(m_1x_1+m_2qx_2)Q(x_1,\,x_2),\\
C_2(x_1,\,x_2,\,x_3,\,x_4)&=((m_2x_1-(m_1+pm_2)x_2)Q(x_1,\,x_2).\\
 \end{aligned} 
\label{sextsubs1ident}
\end{equation}

Let us now consider  the equation,
\begin{equation}
S(x_1,\,x_2,\,x_3,\,x_4)=0, \label{sexteqngen}
\end{equation}
where $S(x_1,\,x_2,\,x_3,\,x_4)$ is a sextic form defined by,
\begin{multline}
S(x_1,\,x_2,\,x_3,\,x_4)=a_1C_1^2(x_i)+a_2C_1(x_i)C_2(x_i)+a_3C_2^2(x_i)\\
+Q(x_3,\,x_4)[x_2\{a_4C_1(x_i)+a_5C_2(x_i)\}
+x_4\{a_6C_1(x_i)+a_7C_2(x_i)\}]\\
+(a_8x_1^2+a_9x_2^2+a_{10}x_2x_3+a_{11}x_2x_4+a_{12}x_3^2+a_{13}x_3x_4+a_{14}x_4^2)Q^2(x_3,\,x_4),
\label{sextgen}
\end{multline}
with $a_j,\;j=1,\,2,\,\ldots,\,14$, being arbitrary integer coefficients.

We note that $x_1$ occurs only in even degrees in the forms $Q(x_1,\,x_2), C_1(x_i)$ and $C_2(x_i)$, and it is easily observed that this is also true for the form $S(x_i)$. It follows that $S(x_1,\,x_2,\,x_3,\,x_4)=S(-x_1,\,x_2,\,x_3,\,x_4)$. Thus, the form $S(x_1,\,x_2,\,x_3,\,x_4)$ satisfies the property ${\bf D_2}$. Further, in view of the relations \eqref{sextsubs1ident}, it is also easily observed that when we substitute the values of $x_3,\,x_4$ given by \eqref{sextsubs1} in Eq.~\eqref{sexteqngen},  the sextic equation \eqref{sexteqngen} reduces to an equation of type \eqref{deqsubs1} in which $\psi(x_1,\,x_2)=Q^2(x_1,\,x_2)$. 

The parameters $p,\,q$ and the coefficients $a_j$ can easily be chosen such that   the property ${\bf D_1}$ and the remaining conditions stipulated in property ${\bf D_3}$ are  satisfied. Thus, starting from a known solution of Eq.~\eqref{sexteqngen},  we can find a sequence of rational points  on the surface  \eqref{sexteqngen}. 

We note that  Eq.~\eqref{sexteqngen} does not have any terms of the type $x_1^{6-j}x_2^j,\;j=0,\,1,\ldots,\,6$ and accordingly,  the surface represented by \eqref{sexteqngen} has singularities at $(\alpha_1,\,\alpha_2,\,0,\,0)$ for any arbitrary values of $\alpha_1,\,\alpha_2$. We would naturally prefer to construct sextic equations in which all the terms $x_i^6,\;i=1,\,2,\,3,\,4$, are present but all efforts to  do so were futile. 

In the next two subsections, we discuss special cases of Eq.~\eqref{sexteqngen}.

\subsection{}
In Eq.~\eqref{quarteqngen}, we take 
\[
\begin{aligned}
&p&=&1,&\;\; &q&=&-1,&\;\;&a_1&=&3,&\;\;&a_2&=&-1,&\\
&a_3&=&5,&\;\;&a_4&=&44957,&\;\;&a_5&=&6,&\;\;&a_6&=&1,&\\
&a_7&=&2939,&\;\;&a_8&=&29654,&\;\;&a_9&=&13121,&\;\;&a_{10}&=&2,&\\
 &a_{11}&=&-25057,&\;\;&a_{12}&=&-7856,& &a_{13}&=& -8176,&\;\;&a_{14}&=&891,&
\end{aligned}\] 
when we get the following  sextic equation:
\begin{multline}
(3x_3^2-x_3x_4+5x_4^2)x_1^4-(5x_2^2x_3^2+13x_2^2x_3x_4+19x_2^2x_4^2-44957x_2x_3^3\\
-44963x_2x_3^2x_4+44951x_2x_3x_4^2+6x_2x_4^3-29654x_3^4-59309x_3^3x_4\\
+26714x_3^2x_4^2+56370x_3x_4^3-26715x_4^4)x_1^2+7x_2^4x_3^2+23x_2^4x_3x_4\\
+21x_2^4x_4^2-44963x_2^3x_3^3-89932x_2^3x_3^2x_4-6x_2^3x_3x_4^2+44969x_2^3x_4^3\\
+13121x_2^2x_3^4+23302x_2^2x_3^3x_4-21940x_2^2x_3^2x_4^2-29181x_2^2x_3x_4^3\\
+19000x_2^2x_4^4+2x_2x_3^5-25053x_2x_3^4x_4-50116x_2x_3^3x_4^2+25053x_2x_3^2x_4^3\\
+50116x_2x_3x_4^4-25057x_2x_4^5-7856x_3^6-23888x_3^5x_4-7605x_3^4x_4^2\\
+25670x_3^3x_4^3+7605x_3^2x_4^4-9958x_3x_4^5+891x_4^6=0.
\label{sexteqnex1}
\end{multline}

It is easily verified that Eq.~\eqref{sexteqnex1} satisfies the three properties ${\bf D_1}$, ${\bf D_2}$ and ${\bf D_3}$, and a solution of Eq.~\eqref{sexteqnex1}  is given by $(0,\,1,\,0,\,-1)$. 

Now,  starting from the point $P=(0,\,1,\,0,\,-1)$ and  repeatedly applying the $RC$ operation, we   obtain a sequence of points $P,\,(RC)P,\,(RC)^2P,\,\dots,\,$ on the surface \eqref{sexteqnex1}. It turns out that $P$ is a point of  order $11$, and we get just  eleven distinct rational points as follows:
\[
\begin{aligned}
P&=&(0,\,1,\,0,\,-1),\quad &(RC)P&=&(1,\,1,\,-1,\,-2),\\
(RC)^2P&=&(  -2,\,   1,\,   -4,\,   -7), \quad &(RC)^3P&=&(  -1,\,   0,\,   1,\,   3) ,\\
(RC)^4P&=&(  -1,\,   2,\,   3,\,   -1), \quad &(RC)^5P&=&(  -3,\,   2,\,   7,\,   -5) ,\\
(RC)^6P&=& (  3,\,   2,\,   7,\,   -5), \quad &(RC)^7P&=& (  1,\,   2,\,   3,\,   -1),\\
(RC)^8P&=&(  1,\,   0,\,   1,\,   3), \quad &(RC)^9P&=& (  2,\,   1,\,   -4,\,   -7),\\
(RC)^{10}P&=&(  -1,\,   1,\,   -1,\,   -2) , \quad &(RC)^{11}P&=&(  0,\,   1,\,   0,\,   -1) =P.
\end{aligned}
\]

It would be recalled that there cannot be a point of order $11$ on an elliptic curve (see Mazur's theorem \cite[p. 57]{Si}).  It follows from the above example  that the $RC$ operation of generating a sequence of  rational points on a surface defined by an equation of type \eqref{gendeq4var} is  unrelated to the process of finding  a sequence of  rational points on an elliptic curve by repeated addition of a known rational point.

\subsection{} 
As a second example, in Eq.~\eqref{sexteqngen} we take, 
\[
\begin{aligned}
&p&= &1,\;\; &q&= &3,\;\;&a_1&= &44,\;\;&a_2&= &-36,\\
&a_3&= &36,\;\;&a_4&= &120,\;\;&a_5&= &-72,\;\;&a_6&= &0,\\
&a_7&= &0,\;\;&a_8&= &0,\;\;&a_9&= &-9,\;\;&a_{10}&= &-972,\\
&a_{11}&= &-486,\;\;&a_{12}&= &9,\;\;&a_{13}&= &9,\;\;&a_{14}&= &9,
\end{aligned}\]
when we get the sextic equation,
\begin{multline}
(44x_3^2-36x_3x_4+36x_4^2)x_1^4+(300x_2^2x_3^2+12x_2^2x_3x_4\\
+36x_2^2x_4^2+120x_2x_3^3+48x_2x_3^2x_4+288x_2x_3x_4^2-216x_2x_4^3)x_1^2\\
+540x_2^4x_3^2+540x_2^4x_3x_4+324x_2^4x_4^2+432x_2^3x_3^3\\
+648x_2^3x_3^2x_4+1512x_2^3x_3x_4^2+648x_2^3x_4^3-9x_2^2x_3^4\\
-18x_2^2x_3^3x_4-63x_2^2x_3^2x_4^2-54x_2^2x_3x_4^3-81x_2^2x_4^4\\
-972x_2x_3^5-2430x_2x_3^4x_4-7776x_2x_3^3x_4^2-9234x_2x_3^2x_4^3\\
-11664x_2x_3x_4^4-4374x_2x_4^5+9x_3^6+27x_3^5x_4+90x_3^4x_4^2\\
+135x_3^3x_4^3+198x_3^2x_4^4+135x_3x_4^5+81x_4^6=0. \label{sexteqnex2}
\end{multline}

It is easily verified that Eq.~\eqref{sexteqnex2} satisfies the three properties ${\bf D_1}$, ${\bf D_2}$ and ${\bf D_3}$, and a solution of Eq.~\eqref{sexteqnex2}  is given by $(0,\,1,\,1,\,0)$. Thus we may generate a sequence of rational points on the surface \eqref{sexteqnex2} taking $(0,\,1,\,1,\,0)$ as the initial solution. 

On substituting
\begin{equation}
x_3 = m_1x_1+(m_1+3m_2)x_2,\quad x_4 = m_2x_1-m_1x_2,
\end{equation}
Eq.~\eqref{sexteqnex2} reduces, after removing the irreducible factor $(x_1^2+x_1x_2+3x_2^2)^2$, to
\begin{multline}
 (9m_1^6+27m_1^5m_2+90m_1^4m_2^2+135m_1^3m_2^3+198m_1^2m_2^4\\
+135m_1m_2^5+81m_2^6+44m_1^2-36m_1m_2+36m_2^2)x_1^2\\
+(9m_1^6+63m_1^5m_2+180m_1^4m_2^2+423m_1^3m_2^3+540m_1^2m_2^4\\
+567m_1m_2^5+243m_2^6-972m_1^5-2430m_1^4m_2-7776m_1^3m_2^2\\
-9234m_1^2m_2^3-11664m_1m_2^4-4374m_2^5+120m_1^3+48m_1^2m_2\\
+288m_1m_2^2-216m_2^3+36m_1^2+228m_1m_2-180m_2^2)x_1x_2\\
+(9m_1^6+45m_1^5m_2+198m_1^4m_2^2+405m_1^3m_2^3+810m_1^2m_2^4\\
+729m_1m_2^5+729m_2^6-486m_1^5-3888m_1^4m_2-9234m_1^3m_2^2\\
-23328m_1^2m_2^3-21870m_1m_2^4-26244m_2^5-9m_1^4-18m_1^3m_2\\
-63m_1^2m_2^2-54m_1m_2^3-81m_2^4+72m_1^3+504m_1^2m_2\\
+648m_1m_2^2+1296m_2^3+36m_1^2+180m_1m_2+540m_2^2)x_2^2=0. \label{sexteqnex2subs1}
\end{multline}

We will now show that the total number of  all the rational self-conjugate points and the rational invariant points on the surface \eqref{sexteqnex2} is an odd integer, and then apply Lemma~\ref{infinitepts} to prove that there are infinitely many  integer solutions of Eq.~\eqref{sexteqnex2}. 

 If there exist any  rational self-conjugate points on the surface \eqref{sexteqnex2}, there must exist rational numbers $m_1$ and $m_2$ such that Eq.~\eqref{sexteqnex2subs1} has two coincident roots. Thus, the  discriminant of Eq.~\eqref{sexteqnex2subs1} must be 0. The discriminant of Eq.~\eqref{sexteqnex2subs1} is given by $-9(m_1^2+m_1m_2+3m_2^2)^2d(m_ 1,\,m_2)$ where
\begin{multline}
d(m_1,\,m_2)=27m_1^8+108m_1^7m_2+486m_1^6m_2^2+1080m_1^5m_2^3\\
+2457m_1^4m_2^4+3240m_1^3m_2^5+4374m_1^2m_2^6+2916m_1m_2^7\\
+2187m_2^8-2916m_1^6m_2-8748m_1^5m_2^2-34992m_1^4m_2^3\\
-55404m_1^3m_2^4-104976m_1^2m_2^5-78732m_1m_2^6-78732m_2^7\\
-105012m_1^6-315036m_1^5m_2-971388m_1^4m_2^2-1417716m_1^3m_2^3\\
-1759140m_1^2m_2^4-1102788m_1m_2^5-236520m_2^6+48m_1^5\\
+1008m_1^4m_2+3120m_1^3m_2^2+7056m_1^2m_2^3+8208m_1m_2^4\\
+6480m_2^5+26168m_1^4+23760m_1^3m_2+69576m_1^2m_2^2\\
-13104m_1m_2^3-18792m_2^4-1728m_1^3+3888m_1^2m_2\\
+24624m_1m_2^2-66096m_2^3-1776m_1^2+2064m_1m_2\\
-720m_2^2+448m_1+1344m_2+560.\quad \quad \quad 
\label{dissexteqnex2}
\end{multline}
 Since $(m_1^2+m_1m_2+3m_2^2)^2 \neq 0$, we must have $d(m_1,\,m_2)=0$.  As $m_1$ and $m_2$ are rational numbers, we write $m_1=n_1/n_0,\;m_2=n_2/n_0$ where $n_1,\,n_2$ and $n_0$ are integers such that $n_0 \neq 0$ and ${\rm gcd}(n_0,\,n_1,\,n_2)=1$, and now, the condition $d(m_1,\,m_2)=0$ may be written as follows:
\begin{multline}
560n_0^8+448(n_1+3n_2)n_0^7-48(37n_1^2-43n_1n_2+15n_2^2)n_0^6\\
-432(4n_1^3-9n_1^2n_2-57n_1n_2^2+153n_2^3)n_0^5\\
+8(3271n_1^4+2970n_1^3n_2+8697n_1^2n_2^2-1638n_1n_2^3-2349n_2^4)n_0^4\\
+48(n_1^2+n_1n_2+3n_2^2)(n_1^3+20n_1^2n_2+42n_1n_2^2+45n_2^3)n_0^3\\
-36(2917n_1^2+2917n_1n_2+730n_2^2)(n_1^2+n_1n_2+3n_2^2)^2n_0^2\\
-2916n_2(n_1^2+n_1n_2+3n_2^2)^3n_0+27(n_1^2+n_1n_2+3n_2^2)^4=0.
\label{conddissexteqnex2}
\end{multline}

We note that in Eq.~\eqref{conddissexteqnex2}, all terms except the last are even integers, hence $(n_1^2+n_1n_2+3n_2^2)$ must be  an even integer and it readily follows that both $n_1$ and $n_2$ must be even integers. On substituting $n_1=2n_3$ and $n_2=2n_4$ in Eq.~\eqref{conddissexteqnex2}, we get the following equation:
\begin{multline}
35n_0^8+56(n_3+3n_4)n_0^7-12(37n_3^2-43n_3n_4+15n_4^2)n_0^6\\
-216(4n_3^3-9n_3^2n_4-57n_3n_4^2+153n_4^3)n_0^5+8(3271n_3^4\\
+2970n_3^3n_4+8697n_3^2n_4^2-1638n_3n_4^3-2349n_4^4)n_0^4\\
+96(n_3^2+n_3n_4+3n_4^2)(n_3^3+20n_3^2n_4+42n_3n_4^2+45n_4^3)n_0^3\\
-144(2917n_3^2+2917n_3n_4+730n_4^2)(n_3^2+n_3n_4+3n_4^2)^2n_0^2\\
-23328n_4(n_3^2+n_3n_4+3n_4^2)^3n_0+432(n_3^2+n_3n_4+3n_4^2)^4=0.
\label{conddissexteqnex2var}
\end{multline}

We note that in Eq.~\eqref{conddissexteqnex2var}, all terms except the first are even integers, hence $n_0$ must  an even integer. This is a contradiction since ${\rm gcd}(n_0,\,n_1,\,n_2)=1$. Hence Eq.~\eqref{conddissexteqnex2} has no integer solutions. It follows that there are no rational self-conjugate points on the surface \eqref{sexteqnex2}.

Next we determine the number of rational invariant points on the surface \eqref{sexteqnex2}. These points are precisely the rational points on \eqref{sexteqnex2} satisfying the conditions $x_1=0$ and $x_2 \neq 0$. On substituting $x_1=0$ in \eqref{sexteqnex2}, we get the condition,
\begin{multline}
 60x_2^4x_3^2+60x_2^4x_3x_4+36x_2^4x_4^2+48x_2^3x_3^3
+72x_2^3x_3^2x_4+168x_2^3x_3x_4^2\\+72x_2^3x_4^3-x_2^2x_3^4
-2x_2^2x_3^3x_4-7x_2^2x_3^2x_4^2-6x_2^2x_3x_4^3-9x_2^2x_4^4
-108x_2x_3^5\\-270x_2x_3^4x_4-864x_2x_3^3x_4^2-1026x_2x_3^2x_4^3
-1296x_2x_3x_4^4-486x_2x_4^5\\+x_3^6+3x_3^5x_4+10x_3^4x_4^2
+15x_3^3x_4^3+22x_3^2x_4^4+15x_3x_4^5+9x_4^6=0.\label{sextcondRP}
\end{multline}

Eq.~\eqref{sextcondRP} represents a curve of genus 5 in projective space and hence it has a finite number of integer solutions. It is readily verified that if $(x_2,\,x_3,\,x_4)$ $=(\alpha_2,\,\alpha_3,\,\alpha_4)$ is any integer  solution of \eqref{sextcondRP} with $\alpha_2 \neq 0$ and $\alpha_4 \neq 0$, then another distinct integer solution of \eqref{sextcondRP} is given by $(x_2,\,x_3,\,x_4)=(\alpha_2,\,\alpha_3+\alpha_4,\,-\alpha_4)$. Thus, integer solutions of \eqref{sextcondRP} can be paired off except for those solutions in which $x_4=0$. On substituting $x_4=0$ in \eqref{sextcondRP}, we get 
\begin{equation}
x_3^2(x_2-x_3)(60x_2^3+108x_2^2x_3+107x_2x_3^2-x_3^3)=0. \label{sextcondRP1}
\end{equation}
We note that we cannot take $x_3=0$ as a solution of Eq.~\eqref{sextcondRP1} since both $x_3$ and $x_4$ cannot be simultaneously 0. Thus, we get just one solution of \eqref{sextcondRP} with $x_2 \neq 0$ and  $x_4=0$, namely $( x_2,\,x_3,\,x_4)=(1,\,1,\,0)$. This yields the single   rational invariant point $(0,\,1,\,1,\,0)$ on the surface \eqref{sexteqnex2}. As the other rational invariant points occur in pairs, this shows that there are an odd  number of  rational invariant points on the surface \eqref{sexteqnex2}.

As there are no self-conjugate points on the surface \eqref{sexteqnex2}, it follows that the total number of rational self-conjugate points and rational invariant points is an odd integer. It now follows from Lemma~\ref{infinitepts} that there are infinitely many rational points on the surface \eqref{sexteqnex2}.

We now obtain a sequence of rational points on the surface \eqref{sexteqnex2} by repeatedly applying the $RC$ operation starting from the point $(0,\,1,\,1,\,0)$. The first three points of the sequence are 
\[
\begin{aligned}
& & (0,\,1,\,1,\,0),\quad  ( -411,\,  37,\,  37,\,  137),\quad \quad\quad \quad\quad \quad \quad \quad\\
&{\rm and} \quad &( 3112824595430551806, \, 686796656401231307,\quad \quad\quad \quad \\
&&\; -183526740019270303,  1115958479906433472).\quad \quad\quad \quad
\end{aligned}
\]
The coordinates of the next point of the sequence are given by integers consisting of 112 and 113 digits, and are accordingly omitted.

The nonhomogeneous diophantine equation obtained by equating the discriminant of Eq.~\eqref{sexteqnex2subs1} to 0 reduces to the equation,
\begin{equation}
d(m_1,\,m_2)+z^2=0. \label{sexteqnex2diseq}
\end{equation}
The values of $(m_1,\,m_2,\,z)$ for the first three solutions of Eq.~\eqref{sexteqnex2diseq}, corresponding to the three rational points on the surface \eqref{sexteqnex2} obtained above, are given by
\[(0,\,1/3,\,137/3),\quad (-10138/52607,\, -16623/52607,\,50623360597/2767496449)\]
and
\[\begin{aligned}
m_1&=-\dfrac{ 45750198554437655354467688322514}{ 211054640652710654187709543384385},\\[0.2in]
 m_2& = \dfrac{ 65569750657565122906946320140387}{ 211054640652710654187709543384385},\\[0.2in]
z&=(18163941046483611219374374579932547716900051\\& \quad \;\;
58819338555544586339749)/(44544061341044824713\\ & \quad \;\; 534258297259269276961486153154314519861828225).
\end{aligned}
\]

\section{Octic and higher degree equations}\label{octic}
  We now show how to construct a diophantine equation, of arbitrarily high even degree, that can be solved by the method described in Section 2. We will use the quadratic forms $Q(x_1,\,x_2)$, $Q_j(x_1,\,x_2,\,x_3,\,x_4),\;j=1,\,2$ and the quartic form $F(x_1,\,x_2,\,x_3,\,x_4)$ defined by \eqref{Mdefqd} and \eqref{quartgen} to construct a  quaternary form $G(x_1,\,x_2,\,x_3,\,x_4)$ of degree $2d$ where $d$ is an arbitrary integer. 
	
	Let $Q_{j+3}(x_1,\,x_2,\,x_3,\,x_4),\;j=0,\,1,\,\ldots,\,d-1$ be arbitrary quadratic forms in the variables $x_1,\,x_2,\,x_3,\,x_4$ such that $x_1$ occurs only in degree 2 in each of these forms. Let
	\begin{multline}
	G(x_1,\,x_2,\,x_3,\,x_4)=\sum_{j=0}^{d-1}Q^{d-1-j}(x_1,\,x_2)Q^j(x_3,\,x_4)Q_{j+3}(x_i)\\
	+\left\{\sum_{j=0}^{d-2}b_jQ^{d-2-j}(x_1,\,x_2)Q^j(x_3,\,x_4)\right\}F(x_1,\,x_2,\,x_3,\,x_4), \quad \quad  \label{defG}
	\end{multline}
	where $b_j,\;j=0,\,1,\,\ldots,\,d-2$ are arbitrary integers. 
	
	We now consider the diophantine equation,
	\begin{equation}
	G(x_1,\,x_2,\,x_3,\,x_4)=0. \label{eqGgen}
	\end{equation}

If  we substitute  the values of $x_3$ and $x_4$ given by \eqref{Msubs1} in Eq.~\eqref{eqGgen}, it follows from the relations \eqref{Mrelqd} and \eqref{quartgen} that  Eq.~\eqref{eqGgen} reduces to an equation of type \eqref{deqsubs1} where $\psi(x_1,\,x_2)$ is, in fact, given by $Q^{d-1}(x_1,\,x_2)$. We also note that if we take $p=0$, then $x_1$ occurs only in even degrees in  the forms $Q(x_1,\,x_2),\,Q_{j+3}(x_i)$ and $F(x_i)$, and therefore $G(x_1,\,x_2,\,x_3,\,x_4)=G(-x_1,\,x_2,\,x_3,\,x_4)$. Finally, we can readily choose the forms $Q_{j+3}(x_i)$, the integers $b_j$  and the arbitrary integer coefficients of the form $F(x_1,\,x_2,\,x_3,\,x_4)$ such that Eq.~\eqref{eqGgen}  has a solution in integers and the remaining conditions stipulated in property ${\bf D_3}$ are  satisfied. 

Thus we can construct a diophantine  equation, of arbitrary degree $2d$,  satisfying all the three properties ${\bf D_1, D_2}$ and ${\bf D_3}$, and we can obtain a sequence of rational solutions starting from a known solution. 

As a numerical example, in \eqref{defG}, we take $d=5,\,p=0,\,q=2,\;Q_3(x_i)=x_1^2+x_2^2,\;Q_4(x_i)=2x_1^2+3x_2^2,\;Q_5(x_i)=Q_6(x_i)=0,\;Q_7(x_i)=-(x_3^2+2x_4^2)$, 
and the coefficients $a_j$ of the form $F(x_1,\,x_2,\,x_3,\,x_4)$ as follows:
\[
\begin{aligned}
a_1 &=& 1,\;\; &a_2 &=& -1,\;\; &a_3 &=& -2,\;\; &a_4 &=& h,\\
a_5 &=& h,\;\; &a_6 &=& -1,\;\; &a_7 &=& -2,\;\; &a_8 &=& 2h,\\
a_9 &=& 0,\;\;  &a_{10} &=& -1,\;\; &a_{11} &=& 0,\;\; &a_{12} &=& -2,
\end{aligned}
\]
where $h$ is an arbitrary parameter.

We thus get the following  equation of degree 10:
\begin{multline}
(x_1^2+2x_2^2)^4(x_1^2+x_2^2)+(x_1^2+2x_2^2)^3(x_3^2+2x_4^2)(2x_1^2+3x_2^2)\\
-(x_3^2+2x_4^2)^5+(x_1^2+2x_2^2)^3\{x_1^4+4x_1^2x_2^2+4x_2^4\\
-4x_2^2x_3^2-8x_2^2x_4^2-x_3^4-4x_3^2x_4^2-4x_4^4\\
+2hx_2(x_1^2x_3-2x_1^2x_4+2x_2^2x_3-4x_2^2x_4+x_3^3+2x_3x_4^2)\}=0. \label{decic}
\end{multline}

It is readily verified that Eq.~\eqref{decic} satisfies the three properties ${\bf D_1}$, ${\bf D_2}$ and ${\bf D_3}$ and for any arbitrary value of $h$, $(x_1,\,x_2,\,x_3,\,x_4)=(1,\,1,\,1,\,1)$ is a solution of Eq.~\eqref{decic}.

We may now repeatedly apply the $RC$ operation and find a sequence of rational points on the surface \eqref{decic}. The coordinates of these points are rational functions of $h$. As in the case of Eq.~\eqref{quarteqnex1}, we will now prove that given an arbitrary positive integer $n$, howsoever large, there exists an integer value of $h$ such that these $n$ points are distinct. The proof  is similar to the one given in case of Eq.~\eqref{quarteqnex1}.

When $h=0$, Eq.~\eqref{decic} reduces to the following equation:
\begin{multline}
(x_1^2+2x_2^2)^4(x_1^2+x_2^2)+(x_1^2+2x_2^2)^3(x_3^2+2x_4^2)(2x_1^2+3x_2^2)\\
-(x_3^2+2x_4^2)^5+(x_1^2+2x_2^2)^3(x_1^4+4x_1^2x_2^2+4x_2^4\\
-4x_2^2x_3^2-8x_2^2x_4^2-x_3^4-4x_3^2x_4^2-4x_4^4)=0. \label{decic1}
\end{multline}
 On substituting 
\begin{equation}
x_3 = m_1x_1+2m_2x_2,\quad  x_4 = m_2x_1-m_1x_2, \label{decic1subs1}
\end{equation}
 Eq.~\eqref{decic1} reduces, after removing the factor $(x_1^2+2x_2^2)^4$ to the following equation:
\begin{multline}
(m_1^2+2m_2^2+1)(m_1^8+8m_1^6m_2^2+24m_1^4m_2^4+32m_1^2m_2^6+16m_2^8\\
-m_1^6-6m_1^4m_2^2-12m_1^2m_2^4-8m_2^6+m_1^4+4m_1^2m_2^2+4m_2^4-2)x_1^2\\
+(2m_1^{10}+20m_1^8m_2^2+80m_1^6m_2^4+160m_1^4m_2^6+160m_1^2m_2^8\\
+64m_2^{10}+2m_1^4+8m_1^2m_2^2+8m_2^4+m_1^2+2m_2^2-3)x_2^2=0. \label{decic1red}
\end{multline}

Let a rational point $P$ on the surface \eqref{decic1} be given by $(3^t,\,3^t,\,\alpha,\,\beta,)$ where $\alpha,\,\beta$ are integers such that $\alpha \equiv 1\; (\bmod \; 3)$ , $\beta \equiv 1\; (\bmod \; 3)$, and $t$ is  a nonnegative integer. The values of $m_1,\,m_2$ corresponding to the point $P$ are given by 
\begin{equation}
m_1= (\alpha-2\beta)/3^{t+1},\;\; m_2= (\alpha+\beta)/3^{t+1}.
\end{equation}

With these values of $m_1, \,m_2$,  we now work out the conjugate $C(P)$ of the point $P$.   It follows from \eqref{decic1red} that
\[
\frac{x_1(P)}{x_2(P)}+\frac{x_1(CP)}{x_2(CP)} = 0.
\]

Since $x_1(P)=x_2(P)=3^k$, we get $-x_1(CP)=x_2(CP)=k$ where $k$ is an arbitrary integer, and we obtain the values of  $x_3(CP),\,x_4(CP)$ using the relations \eqref{decic1subs1}. We thus obtain the point $C(P)$, and on taking the reflection of this point, we  obtain the point $RC(P)$ which is given by $(3^{t+1},\,3^{t+1},\,\alpha^{\prime},\,\beta^{\prime})$ where $\alpha^{\prime} \equiv -1 (\bmod \; 3)$, $\beta^{\prime} \equiv -1 (\bmod \; 3)$. On applying the $RC$ operation once again, we get  the point $(RC)^2(P)$ which is given by $(3^{t+2},\,3^{t+2},\,\alpha^{\prime \prime},\,\beta^{\prime \prime})$ where $\alpha^{\prime \prime} \equiv 1 (\bmod \; 3)$, $\beta^{\prime \prime} \equiv 1 (\bmod \; 3)$ and it now follows by induction that the sequence of points on the surface  \eqref{decic1},  obtained by starting from $(1,\,1,\,1,\,1)$ as the initial point $P_0$, consists of infinitely many distinct rational points. The first four points of the sequence obtained in this manner are $P_0=(1,\,1,\,1,\,1)$, $P_1=(3,\,3,\,5,\,-1)$, $P_2=(3^2,\,3^2,\,1,\,-11)$, $P_3=(3^3,\,3^3,\,-43,\,-13)$.

Now reverting to the original equation, an argument similar to the one given in case of Eq.~\eqref{quarteqnex1} establishes that for any given arbitrary value of $n$, howsoever large, there exist infinitely many  integer values of $h$ such that Eq.~\eqref{decic} has at least $n$ distinct rational solutions.

We note if the repeated application of the $RC$ operation yields an infinite sequence of rational points on a surface defined by an equation of type \eqref{eqGgen} where we take $p=0$, there will  always be a  curve of genus 0 or 1 on such a  surface. To prove this, we observe from the first identity given by \eqref{Mrelqd} that the value of the ratio $(x_3^2+qx_4^2)/(x_1^2+qx_2^2) =m_1^2+qm_2^2$ is the same for the points $P$ and $C(P)$ since any point and its conjugate are obtained from an equation of type \eqref{deqsubs1red} and therefore correspond to the same values of $m_1$ and $m_2$. Moreover, when we take the reflection of the point $C(P)$, then also the ratio  $(x_3^2+qx_4^2)/(x_1^2+qx_2^2)$ remains unchanged. Thus, the ratio $(x_3^2+qx_4^2)/(x_1^2+qx_2^2)$ remains unchanged for any point $P$ and $RC(P)$, and hence it remains constant for all the points of the infinite  sequence of rational points $P,\;RC(P),\;(RC)^2P,\;\ldots$  . Thus, all of these rational points lie both on the surface \eqref{eqGgen} as well as on the surface $(x_3^2+qx_4^2)=k(x_1^2+qx_2^2)$ where $k$ is some constant, and hence they lie also on the intersection of these two surfaces. This intersection is a curve. Thus the infinitely many rational points $(RC)^n(P),\;n=0,\,1,\,2,\,\ldots,\,$ lie on a curve which, in view of Falting's theorem,  must necessarily be a curve of genus 0 or 1.

In fact, for any rational number $h$, all the rational points on the surface \eqref{decic} found by our method lie on the intersection of the surface \eqref{decic} and the quadric $x_1^2+2x_2^2=x_3^2+2x_4^2$, and are given by,
\begin{equation}
\begin{aligned}
x_1&=4hX^3-2X^2+2XY+6hX+2h-1,\\
 x_2& = 4X^3+2hX^2+2hX+2X-Y-h, \\
x_3 &= -4hX^3+10X^2-2XY+2hX+2h-1,\\
 x_4& = 4X^3+6hX^2+2hX-4X+Y+h,
\end{aligned}
\end{equation}
where $(X,\,Y)$  is a rational point on the elliptic curve,
\begin{equation}
Y^2=4(h+1)^2X^4+8h(h-2)X^3+4X^2-4h(h+2)X+(h-1)^2.
\end{equation}
 
We also note that, as compared to Eq.~\eqref{eqGgen},  more general high degree diophantine equations, solvable by our method, can be constructed but Eq.~\eqref{eqGgen} suffices to illustrate  the manner in which this can be done. 
It may, however, be noted that efforts to construct such equations of degree $ \geq 8$ led only to such surfaces on which there is a curve of genus 0 or 1.

\section{Overview of an  extension of the new \\ method}\label{overviewextension}

\subsection{}
We now describe an extension of the method discussed in Section 2 and show how it can be used to solve certain single diophantine equations as well as   diophantine systems consisting of  several simultaneous  equations of the type,
\begin{equation}
f_r(x_1,\,x_2,\,\ldots,\,x_{n})=0,\;r=1,\,2,\,\ldots,\,k,\label{vsysgen} \\
\end{equation}
where $n \geq 5$,  $ k < n-2$ and $f_r(x_i)$ are forms, with integer coefficients,  in the $n$ independent variables $x_1,\,x_2,\,\ldots,\,x_{n}$, and at least one of the forms $f_r(x_i)$ is of degree $d \geq 4$. 

We denote by $V$ the set of all rational solutions of the Eqs~\eqref{vsysgen} and assume that the following properties, analogous to the three properties mentioned in Section 2, are satisfied:

\noindent ${\bf D_1^{\prime}:}$ At least one solution $(\alpha_1,\,\alpha_2,\,\ldots,\,\alpha_{n})$ of Eqs.~\eqref{vsysgen} is known such that $(\alpha_1,\,\alpha_2) \neq (0,\,0)$. 

\noindent ${\bf D_2^{\prime}:}$ There exists a bijective mapping $R: V \rightarrow V$ distinct from the identity mapping. 

\noindent ${\bf D_3^{\prime}:} $  It is possible to set up  $n-k-1$ auxiliary equations 
\begin{equation}
\psi_j(x_1,\,x_2,\, \ldots,\,x_{n},\,m_1,\,\ldots,\,m_{n-k-1})=0, \;j=1,\,2,\,\ldots,\,n-k-1, \label{vaux}
\end{equation}
 such that

\noindent (i) for rational numerical values of $x_1,\,x_2,\, \ldots,\,x_{n}$ such that $(x_1,\,x_2) \neq (0,\,0)$, Eqs.~\eqref{vaux} can be solved to obtain  rational numerical values, not all 0,  of the variables $m_1,\,\ldots,\,m_{n-k-1}$;

\noindent (ii) Eqs.~\eqref{vaux} can be solved together with  the equations $f_r(x_1,\,x_2,\,\ldots,\,x_{n})=$ $ 0,\;r=1,\,2,\,\ldots,\,k-1$ to obtain the values of $x_3,\,x_4,\,\ldots,\,x_n$ in terms of $x_1,\,x_2,\,m_1,\,m_2,\,\ldots,\,m_{n-k-1}$ and on substituting these values,   the last equation $f_k(x_1,\,x_2,\,\ldots,\,x_{n})=0$, reduces to an  equation of the type
\begin{equation}
\phi_0(m_i)x_1^2+\phi_1(m_i)x_1x_2+\phi_2(m_i)x_2^2=0,
\label{sysgen1redf}
\end{equation}
where $\phi_j(m_i),\;j=0,\,1,\,2$ are polynomials in $m_1,\,m_2,\,\ldots,\,m_{n-k-1}$ such that the simultaneous equations $\phi_j(m_i)=0,j=0,\,1,\,2$ do not have any rational solutions.

When the three properties ${\bf D_1^{\prime},\,D_2^{\prime},\, D_3^{\prime}}$ are satisfied, we can obtain a sequence of rational solutions of the simultaneous equations \eqref{vsysgen}.

The method of solving Eqs.~\eqref{vsysgen}  is similar to that of solving Eq.~\eqref{gendeq4var}. Given any point $P$ on the projective variety defined by Eqs.~\eqref{vsysgen}, we can, in general, obtain two new points on the variety --- the point $R(P)$ and the conjugate point $C(P)$ and we can, as before,  perform the $R$ and $C$ operations repeatedly to get new points on the variety. While $R(P)$ is immediately obtained from the definition of the mapping $R$, we will describe below how to find the conjugate of a given point $P$. The properties of the $R$ and $C$ operations and the definition of the order of a point $P$ discussed in Section 2 are also applicable in the case of simultaneous equations in several variables.   

Starting from a known rational point $ P$, with coordinates  $(\alpha_1,\,\alpha_2,\,\ldots,\,$ $ \alpha_{n})$,  on the variety \eqref{vsysgen},  we find the numerical values of $m_1,\,m_2,\,\ldots,\,m_{n-k-1}$ corresponding to the point $P$ by solving the $n-k-1$  equations,
\begin{equation}
\psi_j(\alpha_1,\,\alpha_2,\,\ldots,\,\alpha_n,\,m_1,\,\ldots,\,m_{n-k-1})=0, \;j=1,\,2,\,\ldots,\,n-k-1, \label{vauxalpha}
\end{equation}
and with these values of $m_i$ (which are not all 0),   the quadratic  equation \eqref{sysgen1redf} is necessarily solvable --- one solution corresponding to the known point $P$  is $(x_1,\,x_2)=(\alpha_1,\,\alpha_2)$, while the second solution gives the $x_1$ and $x_2$ coordinates of the conjugate point $C(P)$. These values of $x_1$ and $x_2$ together with the values of $m_1,\,m_2,\,\ldots,\,m_{n-k-1}$, yield the remaining coordinates $x_3,\,x_4,\,\ldots,\,x_n$ of the point $C(P)$. We thus obtain the conjugate point $C(P)$.

We can now obtain the point  $RC(P)$, and on repeatedly applying the $RC$ operation, we obtain a sequence of rational points, 
\begin{equation}
P,\;\;RC(P),\;\;(RC)^2(P),\;\;(RC)^3(P),\ldots,\,\;\;(RC)^k(P),\ldots,\;\;\label{vseqgen}
\end{equation}
on the variety defined by \eqref{vsysgen}. 

We note that when there are several mappings $R: V \rightarrow V$ as happens in the case of symmetric diophantine systems, by using these mappings separately or in combination with each other, we can, starting from a known rational point $P$, generate several sequences of rational points  on the variety defined by \eqref{vsysgen}. 

We have stipulated in property ${\bf D_3^{\prime}}$ that, for arbitrary numerical values of $x_i$ with $(x_1,\,x_2) \neq (0,\,0)$, the auxiliary equations \eqref{vaux}  should yield a nonzero solution for $m_i$ and also that the simultaneous equations $\phi_j(m_i)=0,j=0,\,1,\,2$ should not have any rational solutions.  We now note that we can generate the sequence of rational points \eqref{vseqgen} even when a weaker form of property ${\bf D_3^{\prime}}$ is satisfied by Eqs.~\eqref{vsysgen}. It is sufficient for the  $RC$ operation to  be executed successfully  at every stage if we choose the initial point $P$ and the auxiliary equations in such a manner  that the point $P$ and  the rational points of the sequence \eqref{vseqgen} successively generated at each stage  are such that when we substitute the values of the coordinates $x_i$ of these points in the auxiliary equations \eqref{vaux}, these equations can be solved to obtain rational values, not all 0, of the variables $m_i$, and with these values of $m_i$, the three coefficients $\phi_j(m_i),\, j=0,\,1,\,2$, of the resulting equation \eqref{sysgen1redf} do not vanish simultaneously. 

We note that even if we cannot prove in a certain case that the three coefficients $\phi_j(m_i),\, j=0,\,1,\,2$ in Eq.~\eqref{sysgen1redf}, do not vanish simultaneously, we  may still be able to generate a sequence of rational points on the variety \eqref{vsysgen}. However,  if there actually exist  rational values of $m_i, \,i=1,\,2,\,\ldots,\,$ $ n-k-1,$ such that the three coefficients  $\phi_j(m_i),\, j=0,\,1,\,2$, of Eq.~\eqref{sysgen1redf} vanish simultaneously, Eq.~\eqref{sysgen1redf} is identically satisfied with these values of $m_i$, and  we  get a parametric solution of Eqs.~\eqref{vsysgen}. Further, if we generate a sequence \eqref{vseqgen} of rational points on the variety \eqref{vsysgen} and at some stage, we get a rational point which yields these values   of $m_i$, the sequence terminates. As we are interested in constructing varieties that do not have a parametric solution and on which we can find an arbitrarily large number of rational points, it will be important to prove that there are no rational values of $m_i$ for which the three coefficients $\phi_j(m_i),\, j=0,\,1,\,2$,  vanish simultaneously.

If we can successfully generate an infinite sequence  of distinct rational points on  a given projective variety defined by Eqs.~\eqref{vsysgen} by the method described above, we also simultaneously obtain infinitely many rational solutions of Eq.~\eqref{sysgen1redf}, and hence also of the diophantine equation
\begin{equation}
\phi_1^2(m_i)-4\phi_0(m_i)\phi_2(m_i)=z^2, \label{vsysgenm}
\end{equation}
since Eq.~\eqref{vsysgenm} must be satisfied for Eq.~\eqref{sysgen1redf} to have rational solutions.

As an example of a  diophantine system of type \eqref{vsysgen}, we mention the system of equations,
\begin{align}
x_1+x_2+x_3&=x_4+x_5+x_6, \label{vex11}\\
x_1^3+x_2^3+x_3^3&=x_4^3+x_5^3+x_6^3, \label{vex13}\\
x_1^4+x_2^4+x_3^4&=x_4^4+x_5^4+x_6^4, \label{vex14}
\end{align}
for which a parametric solution is already known \cite[pp.\ 305-6]{Ch1}. It follows from Eqs.~\eqref{vex11} and \eqref{vex13} that
\begin{equation}
\begin{aligned}
(x_1+x_2-x_6)^3-(x_1^3+x_2^3-x_6^3)&=(x_4+x_5-x_3)^3-(x_4^3+x_5^3-x_3^3),\\
{\rm or,}\quad  3(x_1+x_2)(x_1-x_6)(x_2-x_6)&=3(x_4+x_5)(x_4-x_3)(x_5-x_3),
\end{aligned}
\end{equation}
and hence, if we write two auxiliary equations,
\begin{align}
x_1-x_6+m_1(x_3-x_4)&=0,\label{vex1aux1}\\
x_2-x_6+m_2(x_3-x_5)&=0,\label{vex1aux2}
\end{align}
and solve the four equations \eqref{vex11}, \eqref{vex13}, \eqref{vex1aux1}, \eqref{vex1aux2} for $x_3,\,x_4,\,x_5,\,x_6$, and substitute
these values in Eq.~\eqref{vex14}, we can reduce our diophantine system to  an equation of  type \eqref{sysgen1redf}. We note that in view of the symmetry of Eqs.~\eqref{vex11}, \eqref{vex13} and \eqref{vex14}, there are several choices for the mapping $R$. Thus starting from the parametric  solution of the simultaneous equations \eqref{vex11}, \eqref{vex13} and \eqref{vex14} given in \cite{Ch1}, we can obtain several  sequences of parametric solutions of this diophantine system. 

Sequences of solutions for several other solvable diophantine systems, such as  $x_1^r+x_2^r+x_3^r=x_4^r+x_5^r+x_6^r,\;r=1,\,5$ and $x_1^r+x_2^r+x_3^r=x_4^r+x_5^r+x_6^r,\;r=2,\,6$, may be obtained in a similar manner. As parametric solutions of these diophantine systems are already known (see, for instance, \cite{Br1}, \cite{Br2}, \cite{Ch2}, \cite{Ch3}), we will not obtain any new solutions of these systems. These   examples are mentioned here   just to illustrate the applicability of the general method to various diophantine systems. 

We will now focus  on constructing a  projective variety \eqref{vsysgen} on which we can find  an arbitrarily large number of  rational points and on which  the existence of a curve of genus 0 or 1   is not certain. In the next subsection we will  construct  projective varieties defined by equations of type \eqref{vsysgen} in which one of the forms $f_r(x_i)$ is of arbitrarily high degree while the other forms $f_r(x_i)$ are all linear forms.

\subsection{}

It will be convenient to consider projective varieties in $2n$ independent variables. We will accordingly consider varieties defined by equations of the type
\begin{align}
f(x_1,\,x_2,\,\ldots,\,x_{2n})&=0,\label{sysgen1} \\
L_j(x_1,\,x_2,\,\ldots,\,x_{2n})&=0,\;\;j=1,\,2,\ldots,\,n-2,\label{sysgenj}
\end{align}
where $f(x_i)$ is a form of degree $d \geq 4$  and $L_j(x_i),\;j=1,\,2,\ldots,\,n-2,$ are $n-2$ linear  forms in the $2n$ variables $x_1,\,x_2,\,\ldots,\,x_{2n}$.

We will  illustrate  how we can construct systems of simultaneous  equations \eqref{sysgen1} and \eqref{sysgenj} satisfying the three properties ${\bf D_1^{\prime},\,D_2^{\prime}, D_3^{\prime}}$ with the first equation \eqref{sysgen1} being of   arbitrarily high degree. 

We first obtain a composition of forms identity of the type, 
\begin{equation}
\psi(x_1,\,x_2,\,\ldots,\,x_n)\psi(m_1,\,m_2,\,\ldots,\,m_n)=\psi(x_{n+1},\,x_{n+2},\,\ldots,\,x_{2n}) \label{compidentgen}
\end{equation}
where $n \geq 3$ and $\psi(x_1,\,x_2,\,\ldots,\,x_n)$ is a form  in $n$ variables $x_1,\,x_2,\,\ldots,\,x_n$, and the values of the variables $x_{n+1},\,x_{n+2},\,\ldots,\,x_{2n}$ are given in terms of bilinear forms $B_j(m_i,\,x_i), j=1,\,2,\,\ldots,\,n$ in the variables $m_1,\,m_2,\,\ldots,\,m_n$ and $x_1,\,x_2,\,\ldots,\,x_n$. 

One way of obtaining  an identity of type \eqref{compidentgen}  is by   using algebraic integers belonging to the field $\mathbb{Q}(\rho)$  where $\rho$ is a root of the  equation,
\begin{equation}
x^n+p_1x^{n-1}+p_2x^{n-2}+\cdots+p_{n-1}x+p_n=0, \label{algeqn}
\end{equation}
with $p_1,\,p_2,\,\ldots,\,p_n$ being arbitrary rational integers such that Eq.~\eqref{algeqn} is irreducible. We will denote the norm of an algebraic integer $\xi$ by $N(\xi)$. If 
$\alpha,\,\beta,\,\gamma$ are three nonzero integers of the field $\mathbb{Q}(\rho)$ given by
\begin{equation}
\begin{aligned}
\alpha&=x_1+x_2\rho+x_3\rho^2+\cdots +x_n\rho^{n-1},\\
\beta&=m_1+m_2\rho+m_3\rho^2+\cdots +m_n\rho^{n-1},\\
\gamma&=x_{n+1}+x_{n+2}\rho+x_{n+3}\rho^2+\cdots +x_{2n}\rho^{n-1},
\label{defalphabeta}
\end{aligned}
\end{equation}
such that $\alpha \beta = \gamma$, then the values of $x_{n+1},\,x_{n+2},\,\ldots,\,x_{2n}$ are given by 
\begin{equation}
x_{n+j}= B_j(m_1,\,m_2,\,\ldots,\,m_n,\,x_1,\,x_2,\,\ldots,\,x_n),\;\; j=1,\,2,\,\ldots,\,n, \label{valxjbilinear}
\end{equation}
where $B_j(m_i,\,x_i),\; j=1,\,2,\,\ldots,\,n,$ are 
bilinear forms in the variables $m_1,\,m_2,\,$ $ \ldots,\,m_n$ and $x_1,\,x_2,\,$ $\ldots,\,x_n$ and  the  relation $N(\alpha)N(\beta) = N (\gamma)$ yields a composition of forms  identity of the type \eqref{compidentgen} in which $\psi(x_1,\,x_2,\,\ldots,\,x_n)= N(\alpha)$ is a form of  degree $n$.

 Next, for some positive integer $s$,  we obtain $s$ identities of the type,
\begin{equation}
\phi_j(x_1,\,x_2,\,\ldots,\,x_{2n})=\zeta_j(m_1,\,m_2,\,\ldots,\,m_n)\psi(x_1,\,x_2,\,\ldots,\,x_n),\label{sidentgen}
\end{equation}
where $\phi_j(x_1,\,x_2,\,\ldots,\,x_{2n}),\;j=1,\,2,\,\ldots,\,s$, are $s$ forms of degree $n$ in the $2n$ variables $x_1,\,x_2,\,\ldots,\,x_{2n}$ while $\zeta_j(m_i),\;j=1,\,2,\,\ldots,\,s$, are polynomials in $m_1,\,m_2,\,\ldots,\,m_n$ and the identities are  true when the values of $x_{n+1},\,x_{n+2},\,$ $ \ldots,\,$ $ x_{2n}$ are given by \eqref{valxjbilinear}.

We can readily obtain $n$ identities of type \eqref{sidentgen} by considering \eqref{valxjbilinear} as $n$ linear equations in the variables $m_1,\,m_2,\,\ldots,\,m_n$ and solving these equations for $m_1,\,m_2,\,\ldots,\,m_n$, when we get $m_j=\phi_j(x_i)/\psi(x_1,\,x_2,\,\ldots,\,x_n), \;j=1,\,2,\,\ldots,\,n,$ where $\phi_j(x_i)$ are forms of degree $n$ in the $2n$ variables $x_1,\,x_2,\,\ldots,\,$ $ x_{2n}$  and hence we immediately get $n$ identities of the type \eqref{sidentgen}. More identities of type \eqref{sidentgen} can also be obtained. 

We now construct Eq.~\eqref{sysgen1} in which we take the form $f(x_1,\,x_2,\,\ldots,\,x_{2n})$ as follows:
\begin{multline}
f(x_1,\,x_2,\,\ldots,\,x_{2n})=\psi(x_1,\,x_2,\,\ldots,\,x_n)Q_1(x_i)+\psi(x_{n+1},\,\ldots,\,x_{2n})Q_2(x_i)\\
\quad \;\;+ \sum_{j=1}^s\phi_j(x_1,\,x_2,\,\ldots,\,x_{2n})Q_{j+2}(x_i),\quad \quad \quad  \label{deffx}
\end{multline}
where $Q_j(x_i),\;j=1,\,2,\,\ldots,\,s+2$ are arbitrary quadratic forms in the variables $x_1,\,x_2,\,\ldots,\,x_{2n}$. We note that the form $f(x_1,\,x_2,\,\ldots,\,x_{2n})$ is of degree $n+2$ where $n$ is any arbitrary positive integer greater than 2.

We will take Eqs.~\eqref{valxjbilinear} as our $n$ auxiliary equations. When we  substitute the values of $x_{n+1},\,x_{n+2},\,\ldots,\,x_{2n}$  given by \eqref{valxjbilinear} in Eq.~\eqref{sysgen1}, in view of the identities \eqref{compidentgen} and \eqref{sidentgen}, we find that $\psi(x_1,\,x_2,\,$ $ \ldots,\,x_n)$ is a factor of $f(x_1,\,x_2,\,\ldots,\,x_{2n})$. We note that $\psi(x_1,\,x_2,\,\ldots,\,x_n)$, being the norm of a  nonzero algebraic integer, cannot be 0, and hence, on removing this factor,  Eq.~\eqref{sysgen1} reduces to
\begin{multline}
\quad \quad \quad Q_1(x_i)+\psi(m_1,\,m_2,\,\ldots,\,m_n)Q_2(x_i)\\
+\sum_{j=1}^s\zeta_j(m_1,\,m_2,\,\ldots,\,m_n)Q_{j+2}(x_i)=0.\quad \quad \quad \quad \quad  \label{sysgen1red}
\end{multline}
We will reduce this equation further after we have chosen the linear equations \eqref{sysgenj}. 

We will now illustrate how we can choose the form $\psi(x_1,\,x_2,\,\ldots,\,x_n)$ and  the quadratic forms $Q_j(x_i)$ in \eqref{deffx}, as well as  the linear equations \eqref{sysgenj} such that  the  diophantine system given by Eqs.~\eqref{sysgen1} and \eqref{sysgenj}   satisfies the  properties ${\bf D_1^{\prime}}, \,{\bf D_2^{\prime}} $ and ${\bf D_3^{\prime}} $. As a simple example, we  take the $n-2$ linear equations \eqref{sysgenj} as 
\begin{equation} 
x_j=0,\,j=3,\,4,\,\ldots,\,n, \label{lineqnsimple}
\end{equation}
 and the quadratic forms $Q_j(x_i)$ as follows:
\begin{equation}
\begin{aligned}
Q_1(x_i)&=\sum_{i=1}^na_ix_i^2,\quad Q_2(x_i)=\sum_{i=n+1}^{2n}a_ix_i^2,\\
Q_j(x_i)&=0,\; j=3,\,4,\,\ldots,s+2,
\end{aligned}
\label{defqdfms}
\end{equation}
where the coefficients $a_i$ are arbitrary integers such that $a_1 < 0$ and $a_i > 0,\;i=2,\,3,\,\ldots,\,2n$.

On substituting $x_j=0,\,j=3,\,4,\,\ldots,\,n$ in the form $\psi(x_1,\,x_2,\,\ldots,\,x_n)$, we get,
\begin{equation}
\psi(x_1,\,x_2,\,0,\,0,\,\ldots,\,0)=x_1^n+p_1x_1^{n-1}x_2+p_2x_1^{n-2}x_2^2+\cdots+p_nx_2^n.
\end{equation}
We will  take $n$ as even and  choose the integers $p_j$ such that $p_j=0$ for all odd values of  $j$ so that  $\psi(x_1,\,x_2,\,0,\,0,\,\ldots,\,0)=\psi(-x_1,\,x_2,\,0,\,0,\,\ldots,\,0)$. Now \eqref{sysgen1} may be written as 
\begin{multline}
\quad \quad \quad \psi(x_1,\,x_2,\,0,\,0,\,\ldots,\,0)(a_1x_1^2+a_2x_2^2)\\
+\psi(x_{n+1},\,\ldots,\,x_{2n})\left(\sum_{i=n+1}^{2n}a_ix_i^2\right)=0. \quad \quad \quad \quad \quad \label{sysgensimple}
\end{multline}
It follows that  if $(\alpha_1,\,\alpha_2,\,\ldots,\,\alpha_{2n})$  is a solution of our diophantine system, then $(-\alpha_1,\,\alpha_2,\,\ldots,\,\alpha_{2n})$ is also a solution. Thus, we may define  $R$ as the reflection mapping given by 
\begin{equation}
R(\alpha_1,\,\alpha_2,\,\ldots,\,\alpha_{2n})=(-\alpha_1,\,\alpha_2,\,\ldots,\,\alpha_{2n}),
\end{equation}
and now our diophantine system satisfies the property ${\bf D_2^{\prime}} $. 

We have already used the auxiliary equations \eqref{valxjbilinear}  to reduce Eq.~\eqref{sysgen1} to Eq.~\eqref{sysgen1red}. We now solve Eqs.~ \eqref{valxjbilinear} and the linear equations \eqref{lineqnsimple} for $x_3,\,x_4,\,\ldots,\,x_{2n}$ and on substituting the values thus obtained in Eq.~\eqref{sysgen1red}, we get
\begin{equation}
a_1x_1^2+a_2x_2^2+\psi(m_1,\,m_2,\,\ldots,\,m_n)Q_2^{\prime}(x_1,\,x_2)=0, \label{sysgen1red1}
\end{equation}
where we note that $Q_2^{\prime}(x_1,\,x_2)$ is necessarily a positive definite form in the variables $x_1,\,x_2$.   

We have now reduced Eq.~\eqref{sysgen1} to an equation of type \eqref{sysgen1redf}.  We  will  show that Eq.~\eqref{sysgen1red1} cannot be identically 0 for any rational numerical values of $m_1,\,m_2,\,\ldots,\,m_n$ and all values of $x_1$ and $x_2$. If this were to happen when $m_i=\mu_i,\;i=1,\,2,\,\ldots,\,n$, and $\psi(\mu_1,\,\mu_2,\,\ldots,\,\mu_n) > 0$, on taking $(x_1,\,x_2)=(0,\,1)$, the left-hand side of Eq.~\eqref{sysgen1red1} becomes  positive, while if $\psi(\mu_1,\,\mu_2,\,\ldots,\,\mu_n) <  0$, on taking $(x_1,\,x_2)=(1,\,0)$, the left-hand side of Eq.~\eqref{sysgen1red1} becomes negative, and hence, in either case, we have a contradiction. 

We have thus reduced our diophantine system to an equation of type \eqref{sysgen1redf} in which the coefficients of $x_1^2,\, x_1x_2$ and $x_2^2$ do not vanish simultaneously for any rational numerical values of $m_1,\,m_2,\,\ldots,\,m_n$. Thus, our diophantine system  satisfies condition (ii) of property ${\bf D_3^{\prime}} $. 

We now  choose the arbitrary integers $a_i$ such that Eqs.~\eqref{sysgen1} and \eqref{sysgenj} have a solution $(\alpha_1,\,\alpha_2,\,\ldots,\,\alpha_{2n})$ where  $(\alpha_1,\,\alpha_2) \neq (0,\,0)$ and  also  $(\alpha_{n+1},\,$ $ \alpha_{n+2},\,\ldots,\,\alpha_{2n}) \neq (0,\,0,\,\,\ldots,\,0)$.

Now starting from the  known point $P=(\alpha_1,\,\alpha_2,\,\ldots,\,\alpha_{2n})$ on the variety defined by Eqs.~\eqref{sysgen1} and \eqref{sysgenj}, we first find the values of $m_1,\,m_2,\,\ldots,\,m_n$ by solving the equations
\begin{equation}
\alpha_{n+j}= B_j(m_i,\,\alpha_i),\;\; j=1,\,2,\,\ldots,\,n. \label{valxjbilineareqm}
\end{equation} 

We note that for rational numerical values of $x_1,\,x_2,\,\ldots,\,x_n$, not all 0, the linear forms $B_j(m_i,\,x_i),\;\; j=1,\,2,\,\ldots,\,n$, are  linearly independent since $N(\alpha) \neq 0$. Thus, Eqs.~\eqref{valxjbilineareqm} are solvable, and since $(\alpha_{n+1},\,\alpha_{n+2},\,\ldots,\,\alpha_{2n}) \neq (0,\,0,\,\,\ldots,\,0)$, we obtain a solution in which the values of $m_1,\,m_2,\,\ldots,\,m_n$  are not all 0. With these values of $m_i$, Eq.~\eqref{sysgen1red1}  has two solutions corresponding to the point $P$ and its conjugate $C(P)$. We thus get a solution for the $x_1$ and $x_2$ coordinates of the point $C(P)$ such that $(x_1,\,x_2) \neq (0,\,0)$. For all points on the variety, we have $x_j=0,\;j=3,\,4,\,\ldots,\,x_n$, and now we use the relations \eqref{valxjbilinear} to get the remaining coordinates of the point $C(P)$. We now note that for rational numerical values of $m_1,\,m_2,\,\ldots,\,m_n$, not all 0, the linear forms $B_j(m_i,\,x_i),\;\; j=1,\,2,\,\ldots,\,n$, are  linearly independent since $N(\beta) \neq 0$. Since $(x_1,\,x_2) \neq (0,\,0)$ and the values of $m_i$ are also not all 0, it follows that the coordinates $x_{n+1},\,x_{n+2},\,\ldots,\,x_{2n}$ of the point $C(P)$ cannot be simultaneously 0. Now on taking the reflection of $C(P)$, we get the point $RC(P)$ which is such that its  $x_1,\,x_2$ coordinates are not simultaneously 0 and the last $n$ coordinates, $x_{n+1},\,x_{n+2},\,\ldots,\,x_{2n}$, are also not simultaneously 0. 

Thus when we take a point $P$ whose $x_1,\,x_2$ coordinates are not simultaneously 0 and whose last $n$ coordinates  are also not simultaneously 0, the auxiliary equations \eqref{valxjbilinear} can be solved to obtain rational numerical values, not all 0,  of $m_1,\,m_2,\,\ldots,\,m_n$ and on applying the $RC$ operation, we get the point $RC(P)$ whose coordinates satisfy the same conditions. We can therefore repeatedly apply  the $RC$ operation, and the rational points successively generated at each stage satisfy the same conditions as the coordinates of the initial point $P$, and thus the weaker form of condition (i) of property  ${\bf D_3^{\prime}} $ is satisfied.

Thus starting from the point $P$, we may apply the $RC$ operation any number of times to obtain a sequence of rational points,
\[P,\;\;RC(P),\;\;(RC)^2(P),\;\;(RC)^3(P),\ldots,\,\;\;(RC)^k(P),\ldots\]
on the variety defined by Eqs.~\eqref{sysgen1} and \eqref{sysgenj}.

 Each of these points represents a solution not only of Eqs.~\eqref{sysgen1} and \eqref{sysgenj} but also of the single diophantine equation of degree $n+2$ in  $n+2$ variables  obtained by eliminating the $n-2$ variables $x_3,\,x_4,\,\ldots,\,x_n$ from Eqs.~\eqref{sysgen1} and \eqref{sysgenj}. This diophantine equation is given by Eq.~\eqref{sysgensimple}.  

Thus, for any arbitrary even integer $n$ howsoever large, we have  constructed a  projective variety defined by an  equation of degree $n+2$ in $n+2$ variables such that starting from a known point $P$ on the variety, we can generate  an arbitrarily long  sequence of rational points on the variety. It is  expected that, in general, this sequence consists of infinitely  many distinct rational points.
 
As a numerical example, taking $\rho$ as a root of the equation $x^6+2=0$, we get a composition of forms identity of type Eq.~\eqref{compidentgen} in which
\begin{multline}
\psi(x_1,\,x_2,\,\ldots,\,x_6)=x_1^6+(12x_2x_6+12x_3x_5+6x_4^2)x_1^4\\
-(12x_2^2x_5+24x_2x_3x_4+4x_3^3-24x_3x_6^2-48x_4x_5x_6-8x_5^3)x_1^3\\
+(12x_2^3x_4+18x_2^2x_3^2+36x_2^2x_6^2-72x_2x_4x_5^2-72x_3^2x_4x_6\\
+36x_3^2x_5^2+12x_4^4+48x_4x_6^3+72x_5^2x_6^2)x_1^2\\
-(12x_2^4x_3+48x_2^3x_5x_6-72x_2^2x_4^2x_5-48x_2x_3^3x_6\\
+48x_2x_3x_4^3-96x_2x_3x_6^3+96x_2x_5^3x_6+24x_3^4x_5-24x_3^3x_4^2\\
+144x_3x_4^2x_6^2-48x_3x_5^4-96x_4^3x_5x_6+48x_4^2x_5^3-96x_5x_6^4)x_1\\
+2x_2^6+24x_2^4x_4x_6+12x_2^4x_5^2-24x_2^3x_3^2x_6-48x_2^3x_3x_4x_5\\
-8x_2^3x_4^3+16x_2^3x_6^3+24x_2^2x_3^3x_5+36x_2^2x_3^2x_4^2\\
-144x_2^2x_3x_5x_6^2+72x_2^2x_4^2x_6^2+24x_2^2x_5^4-24x_2x_3^4x_4\\
+144x_2x_3^2x_5^2x_6-96x_2x_3x_4x_5^3-48x_2x_4^4x_6+48x_2x_4^3x_5^2\\
+96x_2x_4x_6^4-96x_2x_5^2x_6^3+4x_3^6+24x_3^4x_6^2-96x_3^3x_4x_5x_6\\
-16x_3^3x_5^3+48x_3^2x_4^3x_6+72x_3^2x_4^2x_5^2+48x_3^2x_6^4\\
-48x_3x_4^4x_5-192x_3x_4x_5x_6^3+96x_3x_5^3x_6^2+8x_4^6\\
-32x_4^3x_6^3+144x_4^2x_5^2x_6^2-96x_4x_5^4x_6+16x_5^6+32x_6^6.
\end{multline}

As described above, we construct the following diophantine equation of degree 8 in the 8 variables $x_1,\,x_2,\,x_7,\,x_8,\,\dots,\,x_{12}$:
\begin{multline}
(x_1^6+2x_2^6)(-x_1^2+33x_2^2)\\
=\psi(x_7,\,x_8,\,\ldots,\,x_{12})(x_7^2+x_8^2+x_9^2+x_{10}^2+x_{11}^2+3x_{12}^2). \label{gennumex1}
\end{multline}

Starting from the point $P$ given by $(x_1,\,x_2,\,x_7,\,x_8,\,x_9,\,x_{10},\,x_{11},\,x_{12})=(1,\,1,\,0,\,0,\,0,\,0,\,0,\,1)$, we get an arbitrarily long sequence of rational points on the variety \eqref{gennumex1}. The next two points $RC(P)$ and $(RC)^2P$ of this sequence are given by $(-19,\, 109, \, -60, \, 60, \, -60,\, 60, \, -60, \, 79)$ and 
\[\begin{aligned}
(& -5213798318716593857204393, & 5565884623111221678035287,\\
 & -163345550361845074758840, & 174763856538766110586320, \\
& -109258836892219115576040,&  485050791706620297477120, \\
& 1670808317491786482902760, & 5483216674456163821019527).
\end{aligned}\]
The next point $(RC)^3P$ involves integers consisting of 251 digits and is therefore omitted. There seems to be no doubt that the above sequence of rational points on the variety \eqref{gennumex1} contains infinitely many distinct rational points but this remains to be proved. 

In the next two  sections, we give  examples of  projective varieties whose construction is  a little more involved, and  we prove  that there exists an arbitrarily long sequence of distinct rational points on these varieties.

As a natural extension of our earlier notation,  we will denote the individual coordinates of the points $P,\,C(P),\,R(P)$ and $RC(P)$ by $x_i(P),\,x_i(CP)$,  $x_i(RP)$ and $x_i(RCP),\;i=1,\,2,\,\ldots,\,{2n}$, respectively. 

\section{Examples based on  composition of quartic forms}\label{exquarticcomp}
If the  quartic form $F(x_1,\,x_2,\,x_3,\,x_4)$ in four variables $x_1,\,x_2,\,x_3,\,x_4$ is defined by 
\begin{multline}
F(x_1,\,x_2,\,x_3,\,x_4)=x_1^4-2px_1^3x_3+px_1^2x_2^2-(2p^2-4q)x_1^2x_2x_4\\
+(p^2+2q)x_1^2x_3^2+p(p^2-3q)x_1^2x_4^2-4qx_1x_2^2x_3+4pqx_1x_2x_3x_4\\
-2pqx_1x_3^3-2q(p^2-2q)x_1x_3x_4^2+qx_2^4-2pqx_2^3x_4+pqx_2^2x_3^2\\
+q(p^2+2q)x_2^2x_4^2-4q^2x_2x_3^2x_4-2pq^2x_2x_4^3+q^2x_3^4+pq^2x_3^2x_4^2+q^3x_4^4, \label{quartfm}
\end{multline}
we have the identity
\begin{equation}
F(x_1,\,x_2,\,x_3,\,x_4)F(m_1,\,m_2,\,m_3,\,m_4)=F(x_5,\,x_6,\,x_7,\,x_8), \label{quarticcomp}
\end{equation}
where the values of $x_5,\,x_6,\,x_7,\,x_8$ are given by the relations,
\begin{equation}
\begin{aligned}
x_5 &= m_1x_1-qm_4x_2-qm_3x_3-q(m_2-pm_4)x_4,\\
 x_6 &= m_2x_1+m_1x_2-qm_4x_3-qm_3x_4, \\
x_7 &= m_3x_1+(m_2-pm_4)x_2+(m_1-pm_3)x_3\\
&\quad \;\;-\{pm_2-(p^2-q)m_4\}x_4,\\
 x_8 &= m_4x_1+m_3x_2+(m_2-pm_4)x_3+(m_1-pm_3)x_4,
 \label{valx5678quartic}
\end{aligned}
\end{equation}
with $p,\,q$ being arbitrary rational parameters. 

The identity \eqref{quarticcomp} may be obtained, as indicated in Section 6.2, by using norms of algebraic integers belonging to the field $\mathbb{Q}(\rho)$ where $\rho$ is a root of the equation,
\begin{equation}
x^4+px^2+q=0, \label{quarticeqrho}
\end{equation}
where $p,\,q$ are arbitrary rational integers such that Eq.~\eqref{quarticeqrho} is irreducible. 

As mentioned in Section 6.2, we solve the four equations \eqref{valx5678quartic} for $m_1,\,m_2,$ $m_3,\,m_4$ and thus obtain the following identities:
\begin{equation}
\begin{aligned}
G_1(x_i)&=m_1F(x_1,\,x_2,\,x_3,\,x_4),\quad G_2(x_i)&=m_2F(x_1,\,x_2,\,x_3,\,x_4),\\
G_3(x_i)&=m_3F(x_1,\,x_2,\,x_3,\,x_4),\quad G_4(x_i)&=m_4F(x_1,\,x_2,\,x_3,\,x_4),
\end{aligned}
\label{defphiquarticfm}
\end{equation}
where $G_j(x_i),\,j=1,\,2,\,3,\,4$, are quartic forms in the 8 variables $x_1,\,x_2,\,\ldots,\,$ $ x_8$, given by
\begin{equation}
\begin{aligned}
G_1(x_i)&=x_1^3x_5+(qx_2x_8-2px_3x_5+qx_3x_7+qx_4x_6-pqx_4x_8)x_1^2\\& \quad
+(px_2^2x_5-qx_2^2x_7-2qx_2x_3x_6-2p^2x_2x_4x_5+2qx_2x_4x_5\\&\quad
+2pqx_2x_4x_7+p^2x_3^2x_5+qx_3^2x_5-pqx_3^2x_7+2q^2x_3x_4x_8\\&\quad
+p^3x_4^2x_5-2pqx_4^2x_5-p^2qx_4^2x_7+q^2x_4^2x_7)x_1\\
&\quad+q(x_2^3x_6-x_2^2x_3x_5-2px_2^2x_4x_6+qx_2^2x_4x_8+px_2x_3^2x_6\\
&\quad-qx_2x_3^2x_8+2px_2x_3x_4x_5-2qx_2x_3x_4x_7+p^2x_2x_4^2x_6\\
&\quad+qx_2x_4^2x_6-pqx_2x_4^2x_8-px_3^3x_5+qx_3^3x_7-qx_3^2x_4x_6\\
&\quad-p^2x_3x_4^2x_5+qx_3x_4^2x_5+pqx_3x_4^2x_7-pqx_4^3x_6+q^2x_4^3x_8),
\end{aligned}
\end{equation}
\begin{equation}
\begin{aligned}
G_2(x_i)&=x_1^3x_6-(x_2x_5+2px_3x_6-qx_3x_8-qx_4x_7)x_1^2\\
& \quad +(px_2^2x_6-qx_2^2x_8+2px_2x_3x_5-2qx_2x_3x_7-2p^2x_2x_4x_6\\
& \quad +2qx_2x_4x_6+2pqx_2x_4x_8+p^2x_3^2x_6+qx_3^2x_6-pqx_3^2x_8\\
& \quad -2qx_3x_4x_5+p^3x_4^2x_6-2pqx_4^2x_6-p^2qx_4^2x_8+q^2x_4^2x_8)x_1\\
& \quad -px_2^3x_5+qx_2^3x_7-qx_2^2x_3x_6+2p^2x_2^2x_4x_5-qx_2^2x_4x_5\\
& \quad -2pqx_2^2x_4x_7-p^2x_2x_3^2x_5+qx_2x_3^2x_5+pqx_2x_3^2x_7\\
& \quad +2pqx_2x_3x_4x_6-2q^2x_2x_3x_4x_8-p^3x_2x_4^2x_5+q^2x_2x_4^2x_7\\
& \quad +p^2qx_2x_4^2x_7-pqx_3^3x_6+q^2x_3^3x_8+pqx_3^2x_4x_5\\
& \quad -q^2x_3^2x_4x_7+q^2x_3x_4^2x_6-p^2qx_3x_4^2x_6+pq^2x_3x_4^2x_8\\
& \quad +p^2qx_4^3x_5-q^2x_4^3x_5-pq^2x_4^3x_7,
\end{aligned}
\end{equation}
\begin{equation}
\begin{aligned}
G_3(x_i)&=x_1^3x_7-(x_2x_6-px_2x_8+x_3x_5+px_3x_7-px_4x_6\\&
\quad +p^2x_4x_8-qx_4x_8)x_1^2+(x_2^2x_5-2qx_2x_3x_8-2px_2x_4x_5\\&
\quad +2qx_2x_4x_7+px_3^2x_5+qx_3^2x_7-2qx_3x_4x_6+2pqx_3x_4x_8\\&
\quad +p^2x_4^2x_5-qx_4^2x_5-pqx_4^2x_7)x_1-q(-x_2^3x_8+x_2^2x_3x_7\\&
\quad +x_2^2x_4x_6+px_2^2x_4x_8-x_2x_3^2x_6-2x_2x_3x_4x_5-px_2x_4^2x_6\\&
\quad -qx_2x_4^2x_8+x_3^3x_5+qx_3^2x_4x_8+px_3x_4^2x_5-qx_3x_4^2x_7+qx_4^3x_6),\\
\end{aligned}
\end{equation}
\begin{equation}
\begin{aligned}
G_4(x_i)&=x_1^3x_8-(x_2x_7+x_3x_6+px_3x_8+x_4x_5-px_4x_7)x_1^2\\ & \quad 
+(x_2^2x_6+2x_2x_3x_5-2px_2x_4x_6+2qx_2x_4x_8+px_3^2x_6\\ & \quad 
+qx_3^2x_8-2qx_3x_4x_7+p^2x_4^2x_6-qx_4^2x_6-pqx_4^2x_8)x_1\\ & \quad 
-p^2x_2x_4^2x_5+pqx_2x_4^2x_7-pqx_3x_4^2x_6+pqx_4^3x_5 
+q^2x_3x_4^2x_8 \\ & \quad -q^2x_4^3x_7+2px_2^2x_4x_5-px_2x_3^2x_5
-qx_2^2x_3x_8-qx_2^2x_4x_7\\ & \quad +qx_2x_3^2x_7+2qx_2x_3x_4x_6
-qx_2x_4^2x_5-qx_3^3x_6+qx_3^2x_4x_5-x_2^3x_5.
\end{aligned}
\end{equation}

Using the forms $G_j(x_i),\;j=1,\,2,\,3,\,4$, we construct a form $H(x_i)$ defined as follows:
\begin{equation}
\begin{aligned}
H(x_i)&=(-2x_5x_6+3x_6^2+2px_6x_7-3px_6x_8-2qx_7x_8+3qx_8^2)G_1(x_i)\\
& \quad +(2x_5^2-3x_5x_6-2px_5x_7+3px_5x_8+2qx_7^2-3qx_7x_8)G_2(x_i)\\
 & \quad +2q(x_5x_8-x_6x_7)G_3(x_i)-3q(x_5x_8-x_6x_7)G_4(x_i).
\end{aligned}
\label{defGquarticfm}
\end{equation}

We now write $p=h+3,\,q=h+4$ where $h$ is an arbitrary integer, and consider the diophantine system given by the following three simultaneous equations:
\begin{multline}
\quad \quad \quad \quad F^2(x_1,\,x_2,\,x_3,\,x_4)(-x_1^2+3x_2^2+31x_4^2)\\
+h\{F(x_1,\,x_2,\,x_3,\,x_4)+F(x_5,\,x_6,\,x_7,\,x_8)\}H(x_i)\\
-F^2(x_5,\,x_6,\,x_7,\,x_8)\{x_5-2x_7+2x_8+h(x_6-x_8)\}^2=0, \label{sys1quartic}
\end{multline}
\begin{align}
x_3&=0, \label{sys2quartic}\\
x_4&=x_2.
\label{sys3quartic}
\end{align}
We note that Eq.~\eqref{sys1quartic} is of degree 10 in 8 variables.

It is readily verified that 
\begin{align}
F(x_1,\,x_2,\,0,\,x_4)&=F(-x_1,\,x_2,\,0,\,x_4), \\
H(x_1,\,x_2,\,0,\,x_4,\,\dots,\,x_8)&=H(-x_1,\,x_2,\,0,\,x_4,\,\ldots,\,x_8).
\end{align}
 It now follows  that if 
any  point $P$ represented by the coordinates $(x_1,\,x_2,\,0,\,$ $ x_4,\, x_5,\,x_6,\,x_7,\,x_8)$ lies on the variety defined by Eqs.~\eqref{sys1quartic}, \eqref{sys2quartic} and \eqref{sys3quartic}, then the reflection of $P$ given by $R(P)=(-x_1,\,x_2,\,0,\,x_4,\,x_5,\,x_6,\,x_7,\,x_8)$ also lies on that variety. Thus,  our diophantine system satisfies the property ${\bf D_2^{\prime}} $.

We will choose the four equations given by \eqref{valx5678quartic} as our auxiliary equations. In view of the general proof given in Section 6.2 in the context of the auxiliary equations \eqref{valxjbilinear}, it follows that with our choice of  auxiliary equations, if we take our initial point $P$ such that the $x_1,\,x_2$ coordinates of $P$ are simultaneously not 0 and the last four coordinates of $P$ are also not simultaneously 0, when we repeatedly apply the $RC$ operation, we will successively generate rational points whose coordinates satisfy the same conditions as the coordinates of the point $P$, and hence  the weaker form of condition (i) of property $ {\bf D_3^{\prime}}$ is satisfied. 

 Next, we note that in view of the relations \eqref{quarticcomp}, \eqref{defphiquarticfm} and \eqref{defGquarticfm}, when we substitute the values of $x_5,\,x_6,\,x_7,\,x_8$ given by \eqref{valx5678quartic} in Eq.~\eqref{sys1quartic}, then $F^2(x_1,\,x_2,\,x_3,\,x_4)$ can be factored out, and on further substituting the values of $x_3,\,x_4$ given by Eqs.~\eqref{sys2quartic} and \eqref{sys3quartic}, Eq.~\eqref{sys1quartic} finally reduces to a quadratic equation of type \eqref{sysgen1redf} in which the coefficient of $x_1^2$ is as follows:
\[ 1 + F^2(m_1,\,m_2,\,m_3,\,m_4)\{m_1-2m_3+2m_4+h(m_2-m_4)\}^2\]
We omit the coefficients of $x_1x_2$ and $x_2^2$ as they are cumbersome to write. Since the coefficient of $x_1^2$ is necessarily positive for all rational values of $m_i$ and $h$, it follows that our diophantine system  completely  satisfies  condition (ii) of property $ {\bf D_3^{\prime}}$.

It is  readily verified that for any arbitrary value of $h$, a solution of Eqs.~\eqref{sys1quartic}, \eqref{sys2quartic} and \eqref{sys3quartic} is given by $(3,\,1,\,0,\,1,\,3,\,1,\,0,\,1)$. Thus our diophantine system also satisfies the property ${\bf D_1^{\prime}} $.

We will now show that for infinitely many integer values of $h$, the diophantine system given by   equations \eqref{sys1quartic}, \eqref{sys2quartic} and \eqref{sys3quartic} has  an arbitrarily large number of integer solutions.

We first consider the special case of the projective variety defined by Eqs.~\eqref{sys1quartic}, \eqref{sys2quartic} and \eqref{sys3quartic} when $h=0$.  Now Eq.~\eqref{sys1quartic} reduces on making the substitutions \eqref{valx5678quartic}, removing the factor $F^2(x_1,\,x_2,\,x_3,\,x_4)$ and then substituting the values of $x_3,\,x_4$ given by \eqref{sys2quartic} and \eqref{sys3quartic}, to the following quadratic equation:
\begin{equation}
\begin{aligned}
-x_1^2+34x_2^2 =\phi^2(m_1,\,m_2,\,m_3,\,m_4)(x_1+2x_2)^2, \label{sys3red}
\end{aligned}
\end{equation}
where
\begin{multline}
\phi(m_1,\,m_2,\,m_3,\,m_4)= (m_1^2+m_1m_2-3m_1m_3-5m_1m_4+2m_2^2+2m_2m_3\\
-6m_2m_4+4m_3^2+4m_3m_4+8m_4^2)(m_1^2-m_1m_2-3m_1m_3+5m_1m_4\\
+2m_2^2-2m_2m_3-6m_2m_4+4m_3^2-4m_3m_4+8m_4^2)(m_1-2m_3+2m_4). \label{sys3redphi}
\end{multline}

Now Eq.~\eqref{sys3red} may be written as,
\begin{equation}
\{\phi^2(m_i)+1\}x_1^2+4\phi^2(m_i)x_1x_2+4\{\phi^2(m_i)-34\}x_2^2=0. \label{sys3red1}
\end{equation}
We note that for rational numerical values of $m_1,\,m_2,\,m_3,\,m_4$, the coefficients of $x_1^2$ and $x_2^2$ cannot vanish.  

When we start with the rational point $P_0=(3,\,1,\,0,\,1,\,3,\,1,\,0,\,1)$ and repeatedly apply the $RC$ operation, at each successive step we will get values of $m_1,\,m_2,\,m_3,\,m_4$ such that Eq.~\eqref{sys3red} has rational solutions for $x_1$ and $x_2$. As there are no rational values of $x_1,\,x_2$ for which the left-hand side of Eq.~\eqref{sys3red} can become 0, it follows that  the values of $m_1,\,m_2,\,m_3,\,m_4$ that we get at each successive step are such that $\phi(m_1,\,m_2,\,m_3,\,m_4)\neq 0$.

If $P$ is any rational point on the projective variety under consideration, it follows from Eq.~\eqref{sys3red1} that 
\begin{equation}
\frac{x_1(P)}{x_2(P)} +\frac{x_1(CP)}{x_2(CP)}=-\frac{4\phi^2(m_i)}{\phi^2(m_i)+1}.
\end{equation}
We thus get,
\begin{equation}
\frac{x_1(CP)}{x_2(CP)}=-\frac{x_1(P)}{x_2(P)} -\frac{4\phi^2(m_i)}{\phi^2(m_i)+1}.
\end{equation}
and, on taking the reflection of the conjugate point $C(P)$, we find that,
\begin{equation}
\frac{x_1(RCP)}{x_2(RCP)}=\frac{x_1(P)}{x_2(P)}+\frac{4\phi^2(m_i)}{\phi^2(m_i)+1}. 
\end{equation}
 Since $\phi(m_i) \neq 0$, it follows that 
\begin{equation}
\frac{x_1(RCP)}{x_2(RCP)} > \frac{x_1(P)}{x_2(P)}. 
\end{equation}

Thus, starting from the point $P_0$ and repeatedly applying the $RC$ operation, we get a sequence of rational points such that the ratios $x_1(P_j)/x_2(P_j)$ pertaining to  the successive points $P_j,\;j=0,\,1,\,2,\,\ldots$, form a strictly  increasing monotonic  sequence.  We thus get an infinite sequence of rational points in the special case when $h=0$.

The first four points of the sequence of rational points obtained  in the special case $h=0$ are 
$(3,\,  1,\,  0,\, 1,\, 3,\, 1,\,0,\,  1),\;$ $ (5,\,  1,\,  0,\,  1,\, -5,\,  1,\, 0,\,  1),$
\[(489,\,  87,\,  0,\,  87,\,  841,\, -353,\,0,  -353)\]
 and 
\[(228105,\, 39465,\,  0,\, 39465,\,  -769129, 369377,\, 0,\, 369377).\]

We now revert to the projective  variety defined by Eqs.~\eqref{sys1quartic}, \eqref{sys2quartic} and \eqref{sys3quartic} when $h$ is an arbitrary integer.  Starting from the rational point $P_0$, we may  apply the $RC$ operation $n$ times where $n$ is any arbitrary positive integer howsoever large, and thus obtain a sequence of rational points $P_0,\,P_1,\,P_2,\,\ldots,P_n$ on the variety such that  the coordinates of these points are given by rational functions of $h$. Now,  following the same argument as was used in Section 3.2, it is established that there are infinitely many integer values of $h$ such that there are an arbitrarily large number of  distinct rational points  on the variety defined by the simultaneous equations \eqref{sys1quartic}, \eqref{sys2quartic} and \eqref{sys3quartic}. 

We can readily eliminate the variables $x_3$ and $x_4$ from Eqs.~\eqref{sys1quartic}, \eqref{sys2quartic} and \eqref{sys3quartic}. The resulting equation, which is of degree 10 in 6 variables $x_1,\,x_2,\,x_5,\,x_6,\,x_7,\,x_8$, has an arbitrarily large number of integer solutions for infinitely many integer values of $h$.  As this tenth degree equation is too cumbersome to write, we do not give it explicitly.

	For any arbitrary value of $h$, we note that the ratio 
	\begin{equation}
\frac{ \displaystyle F(x_5,\,x_6,\,x_7,\,x_8)}{ \displaystyle F(x_1,\,x_2,\,x_3,\,x_4)}	=F(m_1,\,m_2,\,m_3,\,m_4), \label{relratquartfms}
\end{equation}
is the same for any point $P$ and its conjugate $C(P)$ (since the values of $m_1,\,m_2,\,m_3,\,$ $ m_4$ are the same for both $P$ and $C(P)$). Further, since $x_3=0$ for all solutions of the simultaneous equations \eqref{sys1quartic}, \eqref{sys2quartic} and \eqref{sys3quartic}, and $F(x_1,\,x_2,\,0,\,x_4)=F(-x_1,\,x_2,\,0,\,x_4)$, the above ratio is also the same for any rational point $P$ on the variety defined by Eqs.~\eqref{sys1quartic}, \eqref{sys2quartic} and \eqref{sys3quartic}  and its reflection $R(P)$. 

It follows that the ratio $F(x_5,\,x_6,\,x_7,\,x_8)/F(x_1,\,x_2,\,x_3,\,x_4)$ is the same for all the points of the infinite sequence of rational points $P_0,\,P_1,\,P_2,\,\ldots$  . For the point $P_0$, this ratio is 1, and hence all the points of the aforementioned infinite sequence  also satisfy the equation,
\begin{equation}
  F(x_1,\,x_2,\,x_3,\,x_4)=F(x_5,\,x_6,\,x_7,\,x_8). \label{sys4quartic}
\end{equation}
Using this relation, we can reduce Eq.~\eqref{sys1quartic} to the following sextic equation:
\begin{multline}
F(x_1,\,x_2,\,x_3,\,x_4)\{-x_1^2+3x_2^2+31x_4^2\\
-(x_5-2x_7+2x_8+h(x_6-x_8))^2\}+2hH(x_i)=0, \label{sys5quartic}
\end{multline}

It follows that  there are infinitely many values of $h$ for which the diophantine system consisting of the four simultaneous equations \eqref{sys2quartic}, \eqref{sys3quartic}, \eqref{sys4quartic} and \eqref{sys5quartic} has an arbitrarily large number of integer solutions. Using the values of $x_3$ and $x_4$ given by \eqref{sys2quartic} and  \eqref{sys3quartic}, we can readily eliminate these two variables and reduce the aforementioned diophantine system to just a pair of simultaneous equations, one of degree 4 and one of degree six, in the six variables  $x_1,\,x_2,\,x_5,\,x_6,\,x_7,\,x_8$.  As the sextic equation is very cumbersome to write, we do not give these two equations explicitly. When $h \neq 0$, if we eliminate $x_1$ from these two equations, 
we get an equation of degree 8 in 5 variables and the possibility of finding an elliptic curve on the projective variety defined by these equations for any arbitrary integer value of $h$ seems rather remote.

It, therefore, appears that there exist  nonzero integer values of $h$ for which the projective varieties defined by  Eqs.~\eqref{sys1quartic}, \eqref{sys2quartic}, \eqref{sys3quartic}, as well as by Eqs.~\eqref{sys2quartic}, \eqref{sys3quartic}, \eqref{sys4quartic} and \eqref{sys5quartic} have  an arbitrarily large number of  rational points but it is unlikely that there is a curve of genus 0 or 1 on either of the two varieties. 

When $h=1$, the first four solutions of the simultaneous equations \eqref{sys2quartic}, \eqref{sys3quartic}, \eqref{sys4quartic} and \eqref{sys5quartic}, obtained as described above, are as follows:
\[(3,\,  1,\,  0,\, 1,\, 3,\, 1,\,0,\,  1), \quad (7,\,  1,\,  0,\, 1,\, -7,\, 1,\,0,\,  1),\]
\[( 2734239, \, 3306073, \, 0, \, 3306073, \,13666439, \, -2392627, \, 4558960, \, -3695187),\]
and
\[
\begin{aligned}
(&-3163872323529222681410246909549439930487861, \\
 &294781492568707070721896411005966791102577, \\
&0, \,294781492568707070721896411005966791102577, \\
&-2921410980830404790669801314286995320934339, \\
 &5201133567144847068196406071233329607771277, \\
 &896748650376707929120182583747883623109040,\\
 & 2508328560483795657430399046587193165809837).
\end{aligned}
\] 

\section{Examples based on composition of cubic forms}\label{excubiccomp} 
If the  cubic form $C(x_1,\,x_2,\,x_3)$ in three variables $x_1,\,x_2,\,x_3$ is defined by 
\begin{multline}
C(x_1,\,x_2,\,x_3)=x_1^3-px_1^2x_2+(p^2-2q)x_1^2x_3+qx_1x_2^2-(pq-3r)x_1x_2x_3\\
-(2pr-q^2)x_1x_3^2-rx_2^3+prx_2^2x_3-qrx_2x_3^2+r^2x_3^3, \label{cubicfm}
\end{multline}
we have the identity,
\begin{equation}
C(x_1,\,x_2,\,x_3)C(m_1,\,m_2,\,m_3)=C(x_4,\,x_5,\,x_6), \label{cubiccomp}
\end{equation}
where the values of $x_4,\,x_5,\,x_6$ are given by the relations,
\begin{equation}
\begin{aligned}
x_4 &= m_1x_1-rm_3x_2-(rm_2-prm_3)x_3,\\
 x_5& = m_2x_1+(m_1-qm_3)x_2-(qm_2-(pq-r)m_3)x_3,\\
 x_6& = m_3x_1+(m_2-pm_3)x_2+(m_1+pm_2-(p^2-q)m_3)x_3\, \label{valx456cubic}
\end{aligned}
\end{equation}
with $p,\,q,\,r$ being arbitrary rational parameters.

The identity \eqref{cubiccomp} may be obtained, as indicated in Section~6.2, by using norms of algebraic integers belonging to the field $\mathbb{Q}(\rho)$ where $\rho$ is a root of the  equation,
\begin{equation}
x^3+px^2+qx+r=0, \label{cubiceqrho} 
\end{equation}
where $p,\,q,\,r$ are arbitrary rational integers such that Eq.~\eqref{cubiceqrho} is irreducible.

We can now solve Eqs.~\eqref{valx456cubic} for $m_1,\,m_2,\,m_3$ to obtain three identities,
\begin{equation}
C_j(x_1,\,x_2,\,\ldots,\,x_6)=m_jC(x_1,\,x_2,\,x_3),\;\; j=1,\,2,\,3, \label{cubicfmident}
\end{equation}
where $C_j(x_i),\;j=1,\,2,\,3,$ are cubic forms in the six variables $x_1,\,x_2,\,\ldots,\,x_6$ and the identities \eqref{cubicfmident} are naturally satisfied when the values of $x_4,\,x_5,\,x_6$ are given by the relations \eqref{valx456cubic}. We can then construct the quintic form,
\begin{multline}
f(x_1,\,x_2,\,\ldots,\,x_6)=C(x_1,\,x_2,\,x_3)Q_1(x_i)+C(x_4,\,x_5,\,x_6)Q_2(x_i)\\
+\sum_{j=1}^3C_j(x_i)Q_{j+2}(x_i).
\end{multline}
where $Q_j(x_i),\;j=1,\,2,\,\ldots,\,5$, are arbitrary quadratic forms in the six variables $x_1,\,x_2,\,\ldots,\,x_6.$ 

On substituting the values of $x_4,\,x_5,\,x_6$ in the form $f(x_i)$, we get $C(x_1,\,x_2,\,$ $ x_3)$ as a factor, and the equation $f(x_i)=0$ reduces, on removing the factor $C(x_1,\,x_2,\,x_3)$, to a quadratic equation. We can thus  construct a pair of simultaneous equations including the quintic equation $f(x_i)=0$ such that these equations satisfy the property  ${\bf D_3^{\prime}}$. If we now impose the condition that  these simultaneous equations also satisfy  the property  ${\bf D_2^{\prime}}$ with our usual definition of reflection, that is, $R(x_1,\,x_2,\,\ldots,\,x_6)=R(-x_1,\,x_2,\,\ldots,\,x_6)$,   then the variable $x_1$ can only appear in even degrees in our quintic equation, and hence  the  term with  $x_1^5$ cannot occur in the equation. While we can solve such equations when all the three properties ${\bf D_1^{\prime}}$, ${\bf D_2^{\prime}}$ and ${\bf D_3^{\prime}}$ are satisfied, we wish to construct a pair of solvable simultaneous equations which includes a  quintic equation in which all the terms $x_i^5,\;i=1,\,\ldots,\,6$, are present, and accordingly we need to modify our definition of reflection. 

We  now prove a lemma that leads to  a new  definition of reflection.

\begin{lemma}\label{defreflection}  If $a$ and $b$ are arbitrary rational numbers such that $a \neq \pm b$, the quadratic form $Q(x_1,\,x_2)$ and the two cubic forms $C_4(x_1,\,x_2)$ and $C_5(x_1,\,x_2)$ defined by
\begin{equation}
\begin{aligned}
Q(x_1,\,x_2)&=(a-b)^2x_1^2+2(a-b)^2x_1x_2+4(a^2+ab+b^2)x_2^2,\\
C_4(x_1,\,x_2)&= x_1^3-\frac{12(a^2+ab+b^2)}{(a-b)^2}x_1x_2^2-\frac{8(a^2+ab+b^2)}{(a-b)^2}x_2^3 \\
C_5(x_1,\,x_2)&=x_1^2x_2+2x_1x_2^2-\frac{4ab}{(a-b)^2} x_2^3,
\label{deffms}
\end{aligned}
\end{equation}
remain unchanged when $x_1$ is replaced by 
\[-ax_1/(a+b)-\{2(a^2+ab+b^2)\}x_2/(a^2-b^2)\] 
and $x_2$ is replaced by \[(a-b)x_1/\{2(a+b)\}-bx_2/(a+b).\]
\end{lemma}

\begin{proof} The lemma is readily verified by direct computation.
\end{proof}

We now consider the simultaneous diophantine equations,
\begin{align}
k_1C(x_1,\,x_2,\,x_3)Q(x_1,\,x_2)&=k_2C(x_4,\,x_5,\,x_6)Q_1(x_4,\,x_5,\,x_6), \label{eqdeg5} \\
x_3&=0, \label{lineqncubicex}
\end{align}
where $k_1,\,k_2$ are arbitrary nonzero integers, $Q_1(x_4,\,x_5,\,x_6)$ is an arbitrary quadratic form in the variables  $x_4,\,x_5,\,x_6$, 
and the form $C(x_1,\,x_2,\,x_3)$ is defined by \eqref{cubicfm} where 
we take the values of the parameters $p, \, q, \,r$ in the cubic form $C(x_1,\,x_2,\,x_3)$ as follows:
\begin{equation}
\begin{aligned}
p&=h,\\
q&=-2\{(h+6)a^2-(2h-6)ab+(h+6)b^2\}/(a-b)^2,\\
r&=4\{2a^2-(h-2)ab+2b^2\}/(a-b)^2,
\end{aligned} \label{valpqrcubicex}
\end{equation}
with $a,\,b$ and $h$ being arbitrary rational parameters such that $a \neq \pm b$.

With these values of $p, \, q,\, r$, we get 
\[C(x_1,\,x_2,\,0)=C_4(x_1,\,x_2)-hC_5(x_1,\,x_2),\]
 and so, on using the linear equation \eqref{lineqncubicex}, we may write Eq.~\eqref{eqdeg5} as follows:
\begin{multline}
k_1\{C_4(x_1,\,x_2)-hC_5(x_1,\,x_2)\}Q(x_1,\,x_2)\\
=k_2C(x_4,\,x_5,\,x_6)Q_1(x_4,\,x_5,\,x_6).\quad \quad \quad \quad  \label{eqdeg5a}
\end{multline}

If  an arbitrary rational point $P$ on the variety defined by Eqs.~\eqref{eqdeg5} and \eqref{lineqncubicex} is given by $(x_1,\,x_2,\,0,\,x_4,\,x_5,\,x_6)$, instead of our usual definition of reflection, we now define the reflection of $P$ as follows:
\begin{equation}
R(P)=(-\frac{\displaystyle ax_1}{\displaystyle a+b}-\frac{\displaystyle 2(a^2+ab+b^2)x_2}{\displaystyle a^2-b^2},\,\frac{\displaystyle (a-b)x_1}{\displaystyle 2(a+b)}-\frac{\displaystyle  bx_2}{\displaystyle a+b},\,0,\,x_4,\,x_5,\,x_6). \label{defRP}
\end{equation} 
It follows from Lemma 2 and Eq.~\eqref{eqdeg5a} that the point $R(P)$ lies on the variety defined by  Eqs.~\eqref{eqdeg5} and \eqref{lineqncubicex}. Thus, the diophantine system given by Eqs.~\eqref{eqdeg5} and \eqref{lineqncubicex} satisfies the property ${\bf D_2^{\prime}}$.

We will choose the three equations given by Eq.~\eqref{valx456cubic} as our auxiliary equations. Here again, in view of the general proof given in Section 6.2 in the context of the auxiliary equations \eqref{valxjbilinear}, it follows that with our choice of  auxiliary equations, if we take our initial point $P$ such that the $x_1,\,x_2$ coordinates of $P$ are simultaneously not 0 and the last three coordinates of $P$ are also not simultaneously 0, when we repeatedly apply the $RC$ operation, we will successively generate rational points whose coordinates satisfy the same conditions as the coordinates of the point $P$, and hence  the weaker form of condition (i) of property $ {\bf D_3^{\prime}}$ is satisfied. 

In view of the identity \eqref{cubiccomp}, on substituting the values of $x_4,\,x_5,\,x_6$ given by \eqref{valx456cubic} in Eq.~\eqref{eqdeg5}, we can factor out $C(x_1,\,x_2,\,x_3)$ and thus reduce Eq.~\eqref{eqdeg5} to a quadratic equation in $x_1$ and $x_2$. We can readily choose the parameters $a,\,b,\,h,\,k_1,\,k_2$ and the quadratic form $Q_1(x_4,\,x_5,\,x_6)$ such that property ${\bf D_1^{\prime}}$ is also satisfied. 

As a specific example, we take $a=1,\, b=-2,\,k_1 = 8,\, k_2 = 39$, when we get the values of $p,\,q,\,r$ from Eq.~\eqref{valpqrcubicex} as
\begin{equation}
p = h,\;\; q = -2h-4,\;\;r = 8h/9+8/3, \label{valpqrspl}
\end{equation}
and we take the quadratic form $Q_1(x_4,\,x_5,\,x_6)$ as given by
\begin{equation}
Q_1(x_4,\,x_5,\,x_6)=(x_4+x_5+x_6)^2+(h+3)(49x_4^2-36x_5^2+x_6^2),
\end{equation}
 and now on using the relation \eqref{lineqncubicex}, Eq.~\eqref{eqdeg5} reduces to the following equation:
\begin{multline}
72(9x_1^3-9hx_1^2x_2-18(h+2)x_1x_2^2-8(h+3)x_2^3)(3x_1^2+6x_1x_2+4x_2^2)\\
 = 13\{ 81x_4^3-81hx_4^2x_5+81(h^2+4h+8)x_4^2x_6-162(h+2)x_4x_5^2\\
+(162h^2+540h+648)x_4x_5x_6+(180h^2+864h+1296)x_4x_6^2\\
-72(h+3)x_5^3+72h(h+3)x_5^2x_6+144(h+3)(h+2)x_5x_6^2\\
+64(h+3)^2x_6^3 \}\{(x_4+x_5+x_6)^2+(h+3)(49x_4^2-36x_5^2+x_6^2)\}.\label{eqdeg5genh}
\end{multline}

It is readily verified that, for all values of $h$, the point $P_0$ given by $(1,\,1,\,0,\,$ $6,\,-7,\,0)$ lies on the variety defined by Eqs.~\eqref{lineqncubicex} and \eqref{eqdeg5genh}. 
 
 On substituting the values of $x_4,\,x_5,\,x_6$ given by \eqref{valx456cubic} where we take $x_3=0$ and the values of $p,\,q,\,r$ as given by \eqref{valpqrspl}, Eq.~\eqref{eqdeg5genh} reduces, on removing the factor  
$9x_1^3-9hx_1^2x_2-18(h+2)x_1x_2^2-8(h+3)x_2^3$,
 to the following quadratic equation:
\begin{multline}
\{1053\psi_0(m_i) \psi_3(m_i)+157464\}x_1^2+\{234\psi_1(m_i) \psi_3(m_i)+314928\}x_1x_2\\
+\{1053\psi_0(m_i) \psi_3(m_i)+13(h+3)\psi_2(m_i)\psi_3(m_i)+209952\}x_2^2=0, \label{redqdcubicex}
\end{multline}
where
\begin{equation}
\begin{aligned}
\psi_0(m_i)&=-(49h+148)m_1^2-2m_1m_2+2m_1m_3+(36h+107)m_2^2\\
& \quad \quad -2m_2m_3-(h+4)m_3^2,\\
\psi_1(m_i)&=-9m_1^2+(324h+954)m_1m_2+(392h^2+2351h+3507)m_1m_3\\
& \quad \quad -9m_2^2+(648h^2+3230h+3840)m_2m_3+(9h^2+26h-12)m_3^2,\\
\psi_2(m_i)&=6885m_1^2+(11664h+23310)m_1m_3-2997m_2^2\\
& \quad \quad +(162h-18)m_2m_3+(8447h^2+27839h+18492)m_3^2,\\
\psi_3(m_i)&=81m_1^3-81hm_1^2m_2+81(h^2+4h+8)m_1^2m_3\\ & \quad \;\;-162(h+2)m_1m_2^2 
+(162h^2+540h+648)m_1m_2m_3+(180h^2\\ & \quad \;\;
+864h+1296)m_1m_3^2-72(h+3)m_2^3+72(h+3)hm_2^2m_3\\ & \quad \;\;+144(h+3)(h+2)m_2m_3^2+64(h+3)^2m_3^3.
\end{aligned}
\label{defpsicubicex}
\end{equation}

We will now show that for infinitely many integer values of $h$, there are an arbitrarily large number of rational points on the projective variety defined by the  simultaneous equations \eqref{lineqncubicex} and \eqref{eqdeg5genh}.  

We first consider the special case when $h=-3$.   It is readily seen that when $h=-3$,  the coefficients of $x_1^2$ and $x_2^2$ in Eq.~\eqref{redqdcubicex} differ by a constant and hence they cannot vanish simultaneously for any values of $m_i$. Thus, condition (ii) of property $ {\bf D_3^{\prime}}$ is satisfied. We now  assume that in this special case a point $P$ on the variety is given by $(1,\,1,\,0,\,\alpha_4,\,\alpha_5,\,\alpha_6)$ where $\alpha_4,\,\alpha_5 $ and $\alpha_6$  are arbitrary rational numbers.  Starting from the point $P$, we may apply the $RC$ operation repeatedly  to obtain a sequence of  rational points $P,\,(RC)P,\,(RC)^2P,\,$ $ \ldots\,$, all of which lie on the variety. The points $(RC)P$ and $(RC)^2P$ of the sequence are given by
\[
(-54,\,135,\,0,\,-162\alpha_4,\,144\alpha_4-90\alpha_5-72\alpha_6,\,-36\alpha_4+36\alpha_5+18\alpha_6)\]
and
\[
(24,\, 24,\,0,\,-288\alpha_4,\, 481\alpha_4+50\alpha_5+52\alpha_6,\,-169\alpha_4-26\alpha_5-28\alpha_6)
\]
respectively. The point $(RC)^2P$ may be written equivalently as $(1,\,1,\,0,\,\alpha_4^{\prime},\,$ $ \alpha_5^{\prime},\,\alpha_6^{\prime})$ where 
\[\alpha_4^{\prime} =-12\alpha_4,\, \alpha_5^{\prime}=(481\alpha_4+50\alpha_5+52\alpha_6)/24,\,\alpha_6^{\prime}=-(169\alpha_4+26\alpha_5+28\alpha_6)/24.\]
  Since $\lvert \alpha_4^{\prime} \rvert > \lvert \alpha_4 \rvert $, it follows that in the special case when $h=-3$, starting from the known point $P_0$ given by $(1,\,1,\,0,\,6,\,-7,\,0)$ and repeatedly applying the $RC$ operation, we  get an infinite sequence of distinct rational points on the variety defined by  Eq.~\eqref{lineqncubicex} and Eq.~\eqref{eqdeg5genh}. The first four points of this sequence are as follows: $(1,\,1,\,0,\, 6,\,-7,\,0)$, $(-18,\,45,\,0,\,-324,\, 498,\,-156)$, $(3,\, 3,\,0,\, -216,\, 317,\, -104)$ and $(-18,\, 45,\,0,\,  3888,\,-5794,\, 1924)$. 
	
	We now revert to the projective variety defined by Eq.~\eqref{lineqncubicex} and Eq.~\eqref{eqdeg5genh} when $h$ is an arbitrary integer. We will first show that there do not exist values of $m_i$ for which Eq.~\eqref{redqdcubicex} is identically satisfied for all values of $x_1,\,x_2$. Accordingly, we consider the following three equations obtained by equating to 0 the coefficients of $x_1^2,\,x_1x_2$ and $x_2^2$ in Eq.~\eqref{redqdcubicex}:
	\begin{align}
	1053\psi_0(m_i) \psi_3(m_i)+157464&=0, \label{cf1cubic}\\
	234\psi_1(m_i) \psi_3(m_i)+314928&=0,\label{cf2cubic}\\
	1053\psi_0(m_i) \psi_3(m_i)+13(h+3)\psi_2(m_i)\psi_3(m_i)+209952&=0 \label{cf3cubic}
	\end{align}
	For these three equations to be satisfied, it is clear that $\psi_3(m_i)$ cannot be 0. Now multiplying Eq.~\eqref{cf1cubic} by 2 and subtracting Eq.~\eqref{cf2cubic}, we get, on removing the factors $(h+3)\psi_3(m_i)$,
	\begin{multline}
	441m_1^2+324m_1m_2+(392h+1175)m_1m_3-324m_2^2\\
	+(648h+1286)m_2m_3+(9h+8)m_3^2=0.\quad \quad \quad  \label{cf12cubic}
	\end{multline}
	Similarly, on multiplying Eq.~\eqref{cf1cubic} by 4 and Eq.~\eqref{cf3cubic} by 3, and taking the difference, we get, 
	\begin{multline}
	(8208h+24651)m_1^2+54m_1m_2+(11664h^2+58302h+69984)m_1m_3\\
	-(3969h+11880)m_2^2+18h(9h+26)m_2m_3\\
	+(8447h^3+53180h^2+102036h+55584)m_3^2=0. \label{cf13cubic}
	\end{multline}
	
	On eliminating $h$ between Eqs.~\eqref{cf12cubic} and \eqref{cf13cubic}, we get,
	\begin{multline}
	391465164951m_1^6+(1055469390948m_2+951172925787m_3)m_1^5\\
	+(397900362636m_2^2+803073780180m_2m_3+700636785345m_3^2)m_1^4\\
	-(1072269163872m_2^3+2176139543544m_2^2m_3+1060172093268m_2m_3^2\\
	+84790818585m_3^3)m_1^3-(856912684464m_2^4+1942353345168m_2^3m_3\\
	+2331720271521m_2^2m_3^2+1181106809600m_2m_3^3+31160975585m_3^4)m_1^2\\
	+(197168862528m_2^5+1012778109936m_2^4m_3+820687816188m_2^3m_3^2\\
	+196155128069m_2^2m_3^3+64601546074m_2m_3^4+1729745180m_3^5)m_1\\
	+241656851520m_2^6+931403339136m_2^5m_3+1368728626356m_2^4m_3^2\\
	+976907330410m_2^3m_3^3+77122752312m_2^2m_3^4\\
	+1429695768m_2m_3^5-708224m_3^6=0. \label{rescf123}
	\end{multline}
	Now Eq.~\eqref{rescf123} represents a curve of genus 4 and it has only finitely many solutions for $m_1,\,m_2,\,m_3$ and it follows from Eq.~\eqref{cf12cubic} that there are only finitely many values of $h$ for which there exist rational values of $m_1,\,m_2,\,m_3$. such that all the three coefficients of Eq.~\eqref{redqdcubicex} vanish simultaneously. Thus when $h$ is arbitrary, we cannot find values of $m_i$ such that the three coefficients of Eq.~\eqref{redqdcubicex} vanish simultaneously. We have thus shown that condition (ii) of property ${\bf D_3^{\prime}}$ is satisfied.
	
Thus when $h$ is arbitrary, starting from the rational point $P_0$, we can  apply the $RC$ operation $n$ times where $n$ is any arbitrary positive integer howsoever large, and obtain a sequence of rational points $P_0,\,P_1,\,P_2,\,\ldots,P_n$  on the variety defined by Eq.~\eqref{lineqncubicex} and Eq.~\eqref{eqdeg5genh}. The coordinates of these points are rational functions of $h$. Now following the same argument as was used in Section 3.2, it follows that there exist infinitely many integer values of $h$ for which  there are an arbitrarily large number of rational points on the projective variety defined by Eqs.~\eqref{lineqncubicex} and \eqref{eqdeg5genh}.
	
	Since Eq.~\eqref{eqdeg5genh} does not contain $x_3$, we have, in effect, shown that for infinitely many values of $h$, the quintic equation in five variables, given by \eqref{eqdeg5genh}, has an arbitrarily large number of solutions.

For any arbitrary value of $h$, we note that the ratio
\begin{equation}
\frac{C(x_4,\,x_5,\,x_6)}{C(x_1,\,x_2,\,x_3)}=C(m_1,\,m_2,\,m_3), \label{eqratcubicfms}
 \end{equation}
is the same for any point $P$ and its conjugate $C(P)$ (since the values of $m_1,\,m_2,\,m_3$ are the same for both $P$ and $C(P)$). Further,  for all rational points $P$ on the variety defined by Eqs.~\eqref{eqdeg5} and \eqref{lineqncubicex}, we have $x_3(P)=0$, and hence the values of both $C(x_1,\,x_2,\,x_3)$ and $C(x_4,\,x_5,\,x_6)$ are the same for any point $P$ and its reflection $R(P)$. Thus, the value of the ratio  $C(x_4,\,x_5,\,x_6)/C(x_1,\,x_2,\,x_3)$ is the same for all points of the  sequence $P_0,\,P_1,\,P_2,\,\ldots$. For the point $P_0$, this ratio is 8, and hence all points of the aforementioned sequence also satisfy the equation
\begin{equation}
8C(x_1,\,x_2,\,x_3)=C(x_4,\,x_5,\,x_6). \label{eqncubicrat}
\end{equation}
Eq.~\eqref{eqncubicrat} reduces, on using Eq.~\eqref{lineqncubicex}, to the following equation:
\begin{multline}
648x_1^3-648hx_1^2x_2-1296(h+2)x_1x_2^2-576(h+3)x_2^3-81x_4^3\\
+81hx_4^2x_5-81(h^2+4h+8)x_4^2x_6+162(h+2)x_4x_5^2\\
-54(3h^2+10h+12)x_4x_5x_6-36(5h^2+24h+36)x_4x_6^2+72(h+3)x_5^3\\
-72h(h+3)x_5^2x_6-144(h+2)(h+3)x_5x_6^2-64(h+3)^2x_6^3=0.\label{eqratcubicfmsex1} 
\end{multline}

We have thus shown that for infinitely many values of $h$, the simultaneous diophantine equations \eqref{eqdeg5genh} and \eqref{eqratcubicfmsex1} have an arbitrarily large number of  solutions. When $h=1$, the first three solutions of the simultaneous equations \eqref{eqdeg5genh} and \eqref{eqratcubicfmsex1} are given by $(1,\,1,\,0,\, 6,\,-7,\,0)$, 
\[(-368765338,\,  605494801,\,  0, \, -297321236,\, -366427558,\, 715340340)\]
and 
\[\begin{aligned}
(&47588476550214311358184089744533539572969269,\\
&-17041771134882562673048141413821106318848779,\,\\
 &0,\, 89001945291963614347958963002771027507707960, \\
& -104068631557829167473628000790179775910721261,\, \\
 &40516749860848893002398908535963291149411960).
\end{aligned}
\]

It follows from Eqs.~\eqref{eqratcubicfms} and \eqref{eqncubicrat} that the values of $m_1,\,m_2,\,m_3$ corresponding to all the rational points $P_0,\,P_1,\,P_2,\,\ldots,\,$ satisfy the condition $ C(m_1,\,m_2,\,m_3)=8$. This condition may be written as follows:
\begin{equation}
\psi_3(m_i)=648. \label{eq1mcubicex}
\end{equation}

Further, the values of $m_1,\,m_2,\,m_3$ corresponding to all the rational points of the infinite sequence, $P_0,\,P_1,\,P_2,\,\ldots\,$, must also satisfy the condition  that the discriminant of Eq. \eqref{redqdcubicex} is  a perfect square. On using the relation \eqref{eq1mcubicex}, this condition may be written as follows:
\begin{multline}
 -13689\psi_0^2(m_i)-169(h+3)\psi_0(m_i)\psi_2(m_i)+169\psi_1^2(m_i)\\
-7371\psi_0(m_i)+702\psi_1(m_i)-39(h+3)\psi_2(m_i)-243=z^2, \label{eq2mcubicex}
\end{multline}
where $z$ is some rational number. 

It follows that there exist infinitely many integer values of $h$ for which the simultaneous equations \eqref{eq1mcubicex} and \eqref{eq2mcubicex} have an arbitrarily large number of rational solutions for $m_1,\,m_2,\,m_3$ and $z$. When $h=1$, the  first two solutions of Eqs.~\eqref{eq1mcubicex} and \eqref{eq2mcubicex}, corresponding to the  first two solutions of Eqs.~\eqref{eqdeg5genh} and \eqref{eqratcubicfmsex1}, are given below:
\[(m_1,\,m_2,\,m_3,\,z)=(50/43,\, 0,\,-117/86,\,13756545/86),
\]
\[
\begin{aligned}
m_1&= -5977631151469496370601034/1446700126228932448001123,\\
 m_2& = 3678131939037546714081930/1446700126228932448001123,\\
 m_3& = 1223704797702532325384490/1446700126228932448001123,\\
z&=105463580688578364176884811517/1446700126228932448001123.
\end{aligned}
\]
The next solution of Eqs.~\eqref{eq1mcubicex} and \eqref{eq2mcubicex}, corresponding to the third solution of Eqs.~\eqref{eqdeg5genh} and \eqref{eqratcubicfmsex1},  involves integers with more than 132 digits and is omitted.

\section{Concluding remarks}
We have shown in this paper that there exist quartic and sextic surfaces, defined by  equations of type \eqref{quarteqngen} and \eqref{sexteqngen} respectively, on which we can find  infinitely many integer points by a new iterative method. It is, however, not easy to determine whether the aforesaid method  can be applied to obtain integer solutions of a  specific  equation of degree 4 or 6 in four variables. It would be of interest to find  criteria by which it is possible to decide whether   the method of this paper is applicable to a given  equation.

We have also shown that there exist projective varieties, defined by equations in several variables, on which there are an arbitrarily large number of integer points. In all such cases, we were able to find these integer points without  finding a curve of genus 0 or 1 on the projective variety under consideration. 

While we have constructed  a few  examples of surfaces and projective varieties defined by high degree equations  and on which there are an arbitrarily large number of  integer points,  it appears that  more general examples of such  surfaces and projective varieties  can be constructed.

The crucial question is whether or not there exists a curve of genus 0 or 1 on  the surfaces and  projective varieties on which we have found an arbitrarily large number of integer points. It would be of considerable interest if it could be proved that a curve of genus 0 or 1 does not lie on these surfaces or projective varieties.  It also needs to be determined whether  these surfaces and varieties are of sufficiently general type.

In the light of the examples already given in this paper and in view of the various possibilities that arise from the  general  method described in Sections 2 and 6, it appears that there may exist  projective varieties, defined by high degree equations, on which there are an arbitrarily large number of integer points and on which  a curve of genus 0  or 1 does not exist.    

\begin{center}
\Large
Acknowledgments
\end{center}
 
I wish to  thank the Harish-Chandra Research Institute, Allahabad for providing me with all necessary facilities that have helped me to pursue my research work in mathematics.

\noindent Postal Address: Ajai Choudhry, 13/4 A Clay Square, \\
\hspace*{1.1in} Lucknow - 226001,\\
\hspace*{1.1in}  INDIA \\
\medskip
\noindent e-mail address: ajaic203@yahoo.com


\begin{thebibliography}{9}

\bibitem{Br1}  A. Bremner, A geometric approach to equal sums of fifth powers,
J. Number Theory \textbf{13} (1981) 337--354.

\bibitem{Br2}  A. Bremner,  A geometric approach to equal sums of sixth powers,
Proc. London Math. Soc. \textbf{43} (1981), 544--581.


\bibitem{Ch1} A. Choudhry, Symmetric diophantine systems, Acta Arithmetica, \textbf{59} (1991), 291--307.

\bibitem{Ch2}   A. Choudhry, On equal sums of fifth powers, Indian Journal of Pure and Applied
Mathematics, \textbf{28} (1997),  1443--1450.

\bibitem{Ch3}  On equal sums of sixth powers, Indian Journal of Pure and Applied
Mathematics, \textbf{25} (1994),  837--841.
\bibitem {Mg} MAGMA online calculator at http://magma.maths.usyd.edu.au/calc/ (accessed on 9 December 2016)

\bibitem{Si} J. H. Silverman and J. Tate,  Rational points on elliptic curves, Springer-Verlag, New York, 1992. 

\end{thebibliography}
\end{document}